\documentclass[11pt]{amsart}
\input xy
\xyoption{all}

\usepackage{amssymb, amsmath, amsthm,amscd,mathabx,amsxtra}
\usepackage{tikz}
\usetikzlibrary{cd,calc,arrows}
\usepackage{mathtools}
\newcommand{\del}{\partial}

\newcommand{\bC}{{\mathbb C}}

\newcommand{\cA}{{\mathcal A}}
\newcommand{\cB}{{\mathcal B}}

\newcommand{\cC}{{\mathcal C}}
\newcommand{\cD}{{\mathcal D}}
\newcommand{\cE}{{\mathcal E}}

\newcommand{\cI}{{\mathcal I}}

\newcommand{\cL}{{\mathcal L}}
\newcommand{\cM}{{\mathcal M}}
\newcommand{\cO}{{\mathcal O}}

\newcommand{\cS}{{\mathcal S}}

\newcommand{\cX}{{\mathcal X}}

\newcommand{\ra}{\rightarrow}
\newcommand{\lra}{\longrightarrow}

\def\mm{\overline{\mathcal{M}}}
\def\ss{\overline{\mathcal{S}}}

\DeclareMathOperator{\Sym}{{Sym}}

\def\mm{\overline{\mathcal{M}}}
\def\ss{\overline{\mathcal{S}}}
\def\rr{\overline{\mathcal{R}}}

\def\thet{\overline{\Theta}_{\mathrm{null}}}

\def\PP{{\textbf P}}

\newtheorem{theorem}{Theorem}[section]
\newtheorem{lemma}[theorem]{Lemma}
\newtheorem{proposition}[theorem]{Proposition}
\newtheorem{corollary}[theorem]{Corollary}

\theoremstyle{definition}
\newtheorem{definition}[theorem]{Definition}
\newtheorem{remark}[theorem]{Remark}

\newtheorem{question}[theorem]{Question}

\begin{document}

\title{Szeg\H{o} kernels and  Scorza quartics on the moduli space of spin curves}
\author[G. Farkas]{Gavril Farkas}

\address{Farkas: Humboldt-Universit\"at zu Berlin, Institut F\"ur Mathematik,  Unter den Linden 6
\hfill \newline\texttt{}
 \indent 10099 Berlin, Germany} \email{{\tt farkas@math.hu-berlin.de}}
\thanks{}

\author[E. Izadi]{Elham Izadi}
\address{Izadi: Department of Mathematics, University of California, San Diego \hfill
\indent \newline\texttt{}
\indent La Jolla, CA 92093-0112, USA}
 \email{{\tt eizadi@math.ucsd.edu}}





\maketitle

\begin{abstract}
We describe an extension at the level of the moduli space $\ss_g^+$ of stable spin curves of genus $g$ of the map associating to an ineffective spin structure its Scorza curve (equivalently, the vanishing locus of the Szeg\H{o} kernel). We compute the class of the Szeg\H{o}-Hodge bundle on $\ss_g^+$, then find a new interpretation, in terms of theta constants, of the Scorza quartic associated to an even spin structure. Our results  describe the superperiod map from the moduli space of supersymmetric curves in the neighborhood of the theta-null divisor and provide a lower bound for the slope of the movable cone of $\ss_g^+$.
\end{abstract}
\setcounter{tocdepth}{1}
\tableofcontents

\section*{Introduction}

To a smooth projective curve $C$ of genus $g$ and to an ineffective even spin structure (that is, a line bundle $\eta\in \mbox{Pic}^{g-1}(C)$ with $\eta^{\otimes 2}\cong \omega_C$  such that $h^0 (C, \eta) = 0$), following \cite{DK}, one can associate the \emph{Scorza correspondence} $S(C, \eta)$  defined as the following curve
\begin{equation}\label{eq:defSC}
S(C, \eta ) := \Bigl\{ (x,y)\in C\times C : H^0 \bigl(C, \eta (x-y)\bigr) \neq  0\Bigr\}.
\end{equation}
Note that $S(C,\eta)$ has a fixed point free involution interchanging the two factors whose quotient is an unramified double cover $q\colon S(C, \eta)\rightarrow T(C, \eta)$, where $T(C, \eta)$ can be viewed as a curve inside the symmetric product $C^{(2)}$.  We denote by  $\ss_g^+$ the moduli space of stable even spin curves of genus $g$ constructed by Cornalba \cite{Cor} and by $\rr_g$ the moduli space of stable Prym curves of genus $g$ parametrizing admissible double covers over stable curves of genus $g$, see \cite{Be,FL}. One can show via \cite{DK, FV} that for a general choice of $[C, \eta]\in \ss_g^+$ the curves $S(C, \eta)$ and $T(C, \eta)$ are smooth and their genera equal to  $g\bigl(S(C, \eta)\bigr)=1+3g(g-1)$ and $g\bigl(T(C, \eta)\bigr)=1+\frac{3}{2}g(g-1)$, respectively.  At the level of moduli spaces, one has the following commutative diagram of Deligne-Mumford stacks
\begin{equation}\label{eq:diagram_stacks}
\xymatrix{
\ss_g^{+} \ar@{.>}[rr]^{\xi} \ar@{.>}[rd]_{\mathfrak{Sc}} &  & \rr_{1+\frac{3}{2}g(g-1)} \ar[dl]^{\chi}&\\
&                         \mm_{1+3g(g-1)}}
\end{equation}
where $\mathfrak{Sc}\bigl([C, \eta]\bigr):=[S(C, \eta)]$ and $\xi\bigl([C, \eta]\bigr):=\bigl[S(C,\eta)\stackrel{q}\longrightarrow  T(C, \eta)\bigr]$, whereas $\chi$ associates to an admissible double cover the stable model of the source curve. Observe that the definition of the curves $S(C, \eta)$ and $T(C, \eta)$ breaks down when either $\eta$ becomes effective, that is, $h^0(C, \eta)\geq 2$, or when the curve $C$ becomes singular. A main goal of this paper is to extend in a modular way the definition of the Scorza map $\mathfrak{Sc}$ to pairs $(C, \eta)$ where $h^0 (C,\eta) =2$ and to the boundary of $\ss_g^+$.

\vskip 3pt

The correspondence $S(C, \eta)$ has been considered classically by Scorza \cite{Sc1, Sc} and used to prove what in modern terms amounts to say
that $\cS_3^+$ is birational to $\cM_3$. A modern exposition of Scorza's work was given by Dolgachev and Kanev \cite{DK}.   Scorza's results were then used by Mukai \cite{mukaiFano3} and Schreyer \cite{Sch} to construct new Fano threefolds.

\vskip 3pt

The Scorza curve $S(C, \eta)$ turns out to be intimately related to the Szeg\H{o} kernel $\mathrm{sz}_{\eta}$ considered in mathematical physics \cite{Fay, HS, TZ}, or in superstring theory \cite{DHP, Wi}. For a spin structure $\eta$ on $C$ with $H^0(C, \eta)=0$, we denote by $\theta_{\eta}$ the function on the Jacobian $JC:=\mbox{Pic}^0(C)$ whose zero locus is the translate of the theta divisor
\begin{equation}\label{eq:theta_eta}
\Theta_{\eta}:=\bigl\{\zeta\in JC: h^0(C, \zeta\otimes \eta)\neq 0\bigr\}.
\end{equation}
Then the \emph{Szeg\H{o} kernel} $\mathrm{sz}_{\eta}$ associated to the spin structure $[C, \eta]$ is the function on $C\times C$
\begin{equation}\label{eq:szego}
\mathrm{sz}_{\eta}(x,y):=\frac{\theta_{\eta}(x-y)}{\theta_{\eta}(0)}.
\end{equation} The Scorza curve $S(C,\eta)$ is the vanishing locus of the Szeg\H{o} kernel, see also (\ref{eq:scorza_diagr}) for details.\footnote{Occasionally, a slightly different definition of the Szeg\H{o} kernel, as an antisymmetric function, is put forward, e.g., \cite{Fay}. This amounts to dividing $\mathrm{sz}_{\eta}(x,y)$ by the prime form $E(x,y)$, which can be regarded as the unique section of the bundle $\cO_{C\times C}(\Delta)$. This does not influence the vanishing locus of $\mathrm{sz}_{\eta}$.}

\vskip 4pt

The Szeg\H{o} kernel is related to the moduli stack $\mathfrak{M}_g^+$ of even supersymmetric curves of genus $g$ considered in \cite{DHP, DW, Wi}. A point in $\mathfrak{M}_g^+$ corresponds to a super curve $\Sigma$ of genus $g$, which determines a smooth algebraic curve $C=\Sigma_{\mathrm{bos}}$ of genus $g$  and an even spin structure $\eta$ on $C$. As explained by D'Hoker and Phong \cite{DHP} and by Witten \cite[Section 8.3]{Wi}, one can define a \emph{super period map}
\begin{equation}\label{eq:sup_period}
P\colon \mathfrak{M}_g^+\dashrightarrow \cA_g
\end{equation}
obtained by integrating the sections of the Berezinian $\mathrm{Ber}(T_{\Sigma})$ of the cotangent bundle of $\Sigma$ using the Szeg\H{o} kernel $\mathrm{sz}_{\eta}$. The period map is undefined along the locus of those $[\Sigma]\in \mathfrak{M}_g$ where the associated bosonic quotient $C$ satisfies $H^0(C, \eta)\neq 0$. Some of our results can be interpreted as describing this period map $P$ when $H^0(C, \eta)\neq 0$.

\vskip 4pt

Inside the spin moduli space $\cS_g^+$ one considers the \emph{theta-null divisor}
$$\Theta_{\mathrm{null}}:=\Bigl\{[C, \eta]\in \cS_g^+: H^0(C, \eta)\neq 0\Bigr\}.$$
It is known \cite{T2} that $\Theta_{\mathrm{null}}$ is an irreducible divisor and its general point corresponds to a spin structure $\eta$ with $h^0(C, \eta)=2$. One has furthermore boundary divisors  $A_i$ and $B_i$ on $\ss_g^+$ and their corresponding divisor classes $\alpha_i=[A_i]$ and $\beta_i=[B_i] \in \mbox{Pic}(\ss_g^+)$ for $i=0, \ldots, \bigl\lfloor \frac{g}{2}\bigr\rfloor$.
The moduli space $\ss_g^+$ being a normal variety, the  rational map $\mathfrak{Sc}\colon \ss_g^+\dashrightarrow \mm_{1+3g(g-1)}$ can be defined along the general point of each divisor in $\ss_g^+$ even though the definition (\ref{eq:defSC}) makes no sense when $h^0(C, \eta)\geq 2$, or when $C$ becomes singular. A priori, this extension need not be modular.

\vskip 4pt

Our first result describes the limiting  Scorza correspondence $\mathfrak{Sc}([C, \eta])\in \mm_{1+3g(g-1)}$,  for a general point $[C, \eta]\in \Theta_{\mathrm{null}}$ corresponding to a pencil $C\stackrel{|\eta|}\rightarrow \PP^1$. Let $\widetilde{C}_{\eta}\longrightarrow C
$ be the double cover branched over the ramification points $x_1, \ldots, x_{4g-4}$ of $|\eta|$.

\vskip 3pt

We introduce the \emph{trace curve} of the pencil $|\eta|$, that is,
$$
\widetilde{\Gamma}_{\eta}:=\Bigl\{(x,y)\in C\times C: H^0\bigl(C, \eta(-x-y)\bigr)\neq 0\Bigr\}.
$$
Observe that $\widetilde{\Gamma}_{\eta}$ intersects the diagonal $\Delta$ of $C\times C$ at the points $(x_1, x_1), \ldots, (x_{4g-4}, x_{4g-4})$. It is easy to show that $\widetilde{\Gamma}_{\eta}$ is a smooth curve of genus $(g-1)(3g-8)+1$, whereas $g(\widetilde{C}_{\eta})=4g-3$.  One has the following result, where we refer to Theorem \ref{propthetanull2limit} for a more precise statement.

\begin{theorem}\label{thm:theta_null4} The limiting Scorza correspondence $S(C, \eta)$ corresponding to  a general vanishing theta-null $[C, \eta]\in \Theta_{\mathrm{null}}$ is the transverse union of $\widetilde{\Gamma}_{\eta}$ and the curve $\widetilde{C}_{\eta}$ meeting at the $(4g-4)$  diagonal points  $(x_1,x_1), \ldots, (x_{4g-4}, x_{4g-4})$.
\end{theorem}

\vskip 4pt

Note that the curve $S(C, \eta)$ described in Theorem \ref{thm:theta_null4} is stable.  The case $g=3$ of Theorem \ref{thm:theta_null4} has been established before by Grushevsky and Salvati Manni \cite{GS}. In Sections \ref{sect:scorza_bdry} and \ref{sect:scorza_Bi} we determine the limiting Scorza curves $S(C, \eta)$ and $T(C, \eta)$ corresponding to a general point of each boundary divisor $A_i$ or $B_i$ of $\ss_g^+$. The most interesting situation occurs along the divisor $B_i$, where $i\geq 1$. We describe our result also in light of the connection pointed out by Donagi and Witten \cite[3.6]{DW} with the moduli space $\mathfrak{M}_{g,1}$ of marked super curves.

\vskip 3pt

A general point of the divisor $B_i$ corresponds to a transverse union $C\cup D$ of two smooth pointed curves $[C, p]$ and $[D, q]$ of genera $i$ and $g-i$ (where  the points $p$ and $q$ get identified) and to two \emph{odd} spin structures $\eta_C$ and $\eta_D$, respectively. We denote by $X_{\eta_C, p}$ the \emph{unique} $\mathbb{Z}_2$-invariant curve inside the linear system $\bigl|\eta_C(p)\boxtimes \eta_C(p)(\Delta)\bigr|$ on $C\times C$, and we have a similar definition for the curve $X_{\eta_D, q}\subseteq D\times D$. We then show that $X_{\eta_C,p}$ and $X_{\eta_D, q}$ are both nodal with a unique singularity at the point $(p,p)$, respectively, $(q,q)$. We denote by $X_{\eta_C,p}'$, respectively, by $X_{\eta_D, q}'$ their normalizations and let $\{p^+, p^-\}$ and $\{q^+, q^-\}$ be the inverse images of the two nodes in $X_{\eta_C,p}$ and $X_{\eta_D,q}$, respectively. The limit Scorza curve has then the following shape (see Theorem \ref{prop:limitBi} for the precise statement):

\begin{figure}[!h]

\vspace*{-27mm}
\setlength{\abovecaptionskip}{10pt minus 5pt}
\hspace{2cm}
{\centering

\includegraphics[width=12cm, height=10cm]{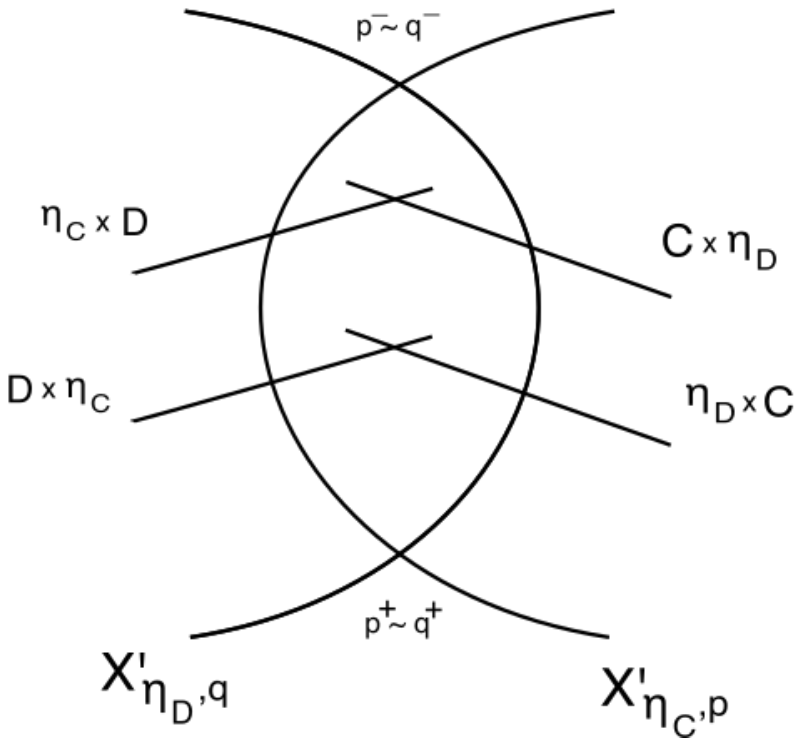}
\vspace*{-28mm}}
\caption{$X_{\eta_C,p}' \ \ \bigcup \ \ \Bigl((\eta_C \times D) \cup (C \times \eta_D)\Bigr) \
\bigcup \ \Bigl((D\times \eta_C)\cup (\eta_D \times C)\Bigr) \ \ \bigcup \ \ X_{\eta_D,p}'.$}
\end{figure}
Note that $X_{\eta_C', p}$ meets transversally $2(g-i-1)$ copies of $D$ attached at the points in the support of $\eta_C$, and a similar statement holds for component $X_{\eta_D', q}$.


\subsection*{The Szeg\H{o}-Hodge class on $\ss_g^+$.}
The Scorza map $\mathfrak{Sc}\colon \ss_g^+ \dashrightarrow \mm_{1+3g(g-1)}$ having been extended as a morphism of \emph{stacks}  along every divisor in $\ss_g^+$, we can attack any problem related to divisor classes involving this map. It is natural to consider the pull-back under this map of the Hodge class $\lambda'\in \mbox{Pic}(\mm_{1+3g(g-1)})$.  We call this class
$$\lambda_{\mathrm{SzH}}:=\mathfrak{Sc}^*(\lambda')$$ the \emph{Szeg\H{o}-Hodge} class on the spin moduli space. We have the following result:

\begin{theorem}\label{thm:szego-hodge-class}
One has the following formula for the Szeg\H{o}-Hodge class on $\ss_g^+$:
$$\lambda_{\mathrm{SzH}}=\frac{77g-25}{4}\cdot\lambda-\frac{69g-21}{16}\cdot \alpha_0-\sum_{i=1}^{\lfloor \frac{g}{2}\rfloor} b_i\cdot \beta_i\in \mathrm{Pic}(\ss_g^+),$$
where $b_i\geq 0$, for $i=1, \ldots, \lfloor \frac{g}{2}\rfloor.$.
\end{theorem}
Note that the coefficients of the boundary classes $\beta_0$ and $\alpha_i$ for $i=1, \ldots, \bigl\lfloor \frac{g}{2}\bigr\rfloor$ are all equal to zero. Theorem \ref{thm:szego-hodge-class} can be compared to a number of similar results in the literature. In \cite{F1} one finds a computation of the  pull-back of the Hodge class under the rational map $\mm_g\dashrightarrow \mm_{g'}$, associating to a curve $C$ of odd genus $g$ the Brill-Noether curve $W^1_{\frac{g+3}{2}}(C)$ consisting of pencils of minimal degree on $C$. In \cite{KKZ, vdGK2} the pull-backs of Hodge classes under various Hurwitz-type correspondences are described.

\vskip 3pt

Theorem \ref{thm:szego-hodge-class} has  applications to the slope of the moduli space $\ss_g^+$. For a projective variety $X$, we denote by $\mathrm{Eff}(X)$ its cone of pseudo-effective divisors and by $\mathrm{Mov}(X)$ its subcone of movable divisors consisting of the classes of those effective $\mathbb Q$-divisors $D$  on $X$ whose stable base locus $\bigcap_{n\geq 1} \mathrm{Bs}|nD|$ has codimension at least $2$ in $X$. For a rational map $f\colon X\dashrightarrow Y$ between normal projective varieties, one has the  inclusion $f^*\bigl(\mathrm{Ample}(Y)\bigr)\subseteq \mathrm{Mov}(X)$.  For an effective divisor $D$ on $\ss_g^+$, we define its slope as the quantity
$$s(D):=\mbox{inf}\Bigl\{\frac{a}{b}\in \mathbb Q_{\geq 0}\cup\{\infty\}: a\cdot \lambda-b\cdot \alpha_0-[D]=\sum_{i=0}^{\lfloor \frac{g}{2} \rfloor}  (a_i\cdot \alpha_i+b_i\cdot \beta_i), \mbox{ where } a_i, b_i\geq 0\Bigr\}.$$
In analogy with the much studied case of the moduli space $\mm_g$, see \cite{CFM} and references therein, we define the \emph{slope}, respectively, the \emph{movable slope} of $\ss_g^+$ as
the following quantities
\begin{equation}\label{eq:slopes}
s\bigl(\ss_g^+\bigr):=\mathrm{inf}\Bigl\{s(D): D\in \mathrm{Eff}\bigl(\ss_g^+\bigr)\Bigr\} \  \mbox{ and }  \
s_{\mathrm{mov}}\bigl(\ss_g^+\bigr):=\mathrm{inf}\Bigl\{s(D):D\in \mathrm{Mov}(\ss_g^+)\Bigr\}.
\end{equation}

Since $s\bigl([\overline{\Theta}_{\mathrm{null}}]\bigr)=4$, see \cite[Theorem 0.2]{F2}, it follows that $s(\ss_g^+)\leq 4$. On the other hand, $\overline{\Theta}_{\mathrm{null}}$ is in general \emph{not} a movable class, as showed in \cite[Theorem 4.2]{FV1}. The class $\lambda_{\mathrm{SzH}}$ being obviously movable on $\ss_g^+$, we obtain the following result:

\begin{theorem}\label{thm:moving_slope}
For any $g\geq 3$, one has the estimate $$s_{\mathrm{mov}}\bigl(\ss_g^+\bigr)\leq 4+\frac{32g-16}{69g-21}.$$
\end{theorem}

\begin{question}\label{question:slope}
Is it true that $s(\ss_g^+)=4$ for every $g\geq 3$? Is it true that
$$\mbox{liminf}_{g\rightarrow \infty} \ s\bigl(\ss_g^+\bigr)=\mbox{liminf}_{g\rightarrow \infty} \ s_{\mathrm{mov}}\bigl(\ss_g^+\bigr)?$$
\end{question}

For the implications in string theory of a positive answer to the first part of Question \ref{question:slope} (formulated as a question on $\mm_g$ rather than on $\ss_g^+$), we refer to \cite{MHNS}.

\subsection*{The Scorza quartic via Wirtinger duality.}
The second topic of this paper is an unconditional new construction, via Wirtinger duality, of the \emph{Scorza quartic hypersurface}
$$F(C, \eta)\in \mbox{Sym}^4 H^0(C, \omega_C)$$ associated to a general pair $[C, \eta]\in \cS_g^+$, where $g\geq 3$. This quartic has been introduced implicitly by Scorza in \cite{Sc} using the theory of apolarity and characterized by the remarkable property that its second polar with respect to pairs of points $(x,y)\in S(C, \eta)$ is a double hyperplane, that is, a quadric of rank one,  see (\ref{eq:def_prop}) for the precise statement. To this day, even the most basic properties of the quartic $F(C, \eta)$, like whether it is smooth for a general $[C, \eta]\in \cS_g^+$, or understanding the relation between them for various choices of spin structures remain wide open.

\vskip 4pt

Remarkably, in  genus $3$, the Scorza quartic defines a plane quartic, different in moduli from $C$, and the assignment
$[C, \eta]\mapsto [F(C,\eta)]$ induces a birational map $\cS_3^+\dashrightarrow \cM_3$, introduced by Scorza in \cite{Sc1}, \cite{Sc} and discussed in modern terms in \cite{DK}. The inverse map $\cM_3\dashrightarrow \cS_3^+$ associates to a plane quartic curve $X\subseteq \PP^2=\PP\bigl(H^0(X, \omega_X)^{\vee}\bigr)$ its \emph{covariant} plane quartic curve
$$C:=\Bigl\{a\in \PP\bigl(H^0(X, \omega_X)\bigr): \mbox{The polar } \ P_a(F) \mbox{ is isomorphic to the Fermat cubic}\Bigr\}.$$
Since $\mbox{dim } PGL(3)=\mbox{dim}\bigl|\cO_{\PP^2}(3)\bigr|-1$, being projectively equivalent to the Fermant cubic is a divisorial condition on the space of all cubics.
Note that $C$ is endowed with a symmetric correspondence $\Xi:=\bigl\{(a,b)\in C\times C: \mbox{rk } P_{a,b}(F)\leq 1\bigr\},$
and one can show that there exists a uniquely determined even spin structure $\eta$ on $C$ such that $\Xi\cong S(C,\eta)$. The assignment $[X]\mapsto [C,\eta]$ induces a rational morphism of stacks $\cM_3\dashrightarrow \cS_3^+$.

\vskip 4pt

We now explain our interpretation of Scorza's quartic in arbitary genus. Having fixed a spin curve $[C, \eta]$ with $h^0(C,\eta)=0$, let  $\phi_C\colon C\hookrightarrow \PP H^0(C, \omega_C)^{\vee}$ be the canonical embedding. For a pair $(x,y)\in S(C,\eta)$, the line bundles $\eta(x-y)$ and $\eta(y-x)$ are effective and there exists a canonical form $h_{x,y}\in H^0(C, \omega_C)$ whose divisor of zeros is
$$\mbox{div}(h_{x,y})=\mbox{div}\bigl(\eta(x-y)\bigr)+\mbox{div}\bigl(\eta(y-x)\bigr)\in C^{(2g-2)}.$$
We consider the polarization map, see \cite[1.1]{DK}
\begin{equation}\label{eq:pol}
P\colon \bigl(H^0(C, \omega_C)^{\vee}\bigr)^{\otimes 2}\otimes \mbox{Sym}^4 H^0(C,\omega_C)\longrightarrow \mbox{Sym}^2 H^0(C,\omega_C).
\end{equation}
For points $x, y\in C$, we write $P_{x,y}=P\bigl(\phi_C(x), \phi_C(y), \bullet\bigr)$, where $\phi_C(x)$ and $\phi_C(y)$ can be interpreted by evaluating canonical forms at $x$, respectively, $y$. Dolgachev and Kanev \cite[Theorem 9.3.1]{DK}, building on \cite{Sc} and assuming  several transversality properties that were ultimately verified generically in \cite{TZ}, showed that there exists a quartic $F(C,\eta)\in \mbox{Sym}^4 H^0(C,\omega_C)$, uniquely determined up to a constant, such that for any pair $(x,y)\in S(C,\eta)$, one has
\begin{equation}\label{eq:def_prop}
P_{x,y}\bigl(F(C,\eta)\bigr)=h_{x,y}^2\in \mbox{Sym}^2 H^0(C,\omega_C).
\end{equation}

\vskip 4pt

The construction of $F(C, \eta)$ remained puzzling, see, e.g., \cite[p. 218]{DK}, where the authors ask: \emph{What are these mysterious quartic hypersurfaces of Scorza?} Our aim is to shed light on this matter and present a simple unconditional  construction of the Scorza quartic in terms of theta functions on $JC$.

\vskip 4pt

We denote by $H^0(JC, 2\Theta_{\eta})_0$ the subspace of $H^0(JC, 2\Theta_{\eta})$ consisting of sections vanishing at the origin. To simplify notation, we set $\Theta:=\Theta_{\eta}$ for the symmetric theta divisor associated to $\eta$. Since $\eta$ is ineffective, the divisor $2\Theta\in \bigl|2\Theta\bigr|$ does not pass through the origin, in particular the restriction map to $\Theta$ induces an isomorphism
\begin{equation}\label{eq:res1}
\mathrm{res}\colon H^0(JC, 2\Theta)_0\stackrel{\cong}\longrightarrow H^0\bigl(\Theta, \cO_{\Theta}(2\Theta)\bigr).
\end{equation}

By taking cohomology in the exact sequence $$0\longrightarrow \cO_{JC}\longrightarrow \cO_{JC}(\Theta)\longrightarrow \cO_{\Theta}(\Theta)\longrightarrow 0,$$ we obtain the following identification
\begin{equation}\label{eq:ident2}
H^0\bigl(\Theta, \cO_{\Theta}(\Theta)\bigr)\cong H^1(JC, \cO_{JC})\cong H^1(C,\omega_C) \cong H^0(C, \omega_C)^{\vee},
\end{equation}
where the middle isomorphism is induced by the Abel-Jacobi map $C\rightarrow JC$.

\vskip 3pt

Furthermore, the difference map $\varphi\colon C\times C\rightarrow JC$ given by $\varphi(x,y):=\cO_C(x-y)$ contracts the diagonal $\Delta$ and factors through the blow-up $\widetilde{JC}$ of $JC$ at the origin $\cO_C\in JC$. Denoting by $\epsilon\colon \widetilde{JC}\rightarrow JC$ the blow-up map, we thus have a morphism $\tilde{\varphi}\colon C\times C\rightarrow \widetilde{JC}$ such that $\varphi=\epsilon\circ \tilde{\varphi}$, see also \cite{welters86}. If $E$ denotes the exceptional divisor on $\widetilde{JC}$, then $$\tilde{\varphi}^*\cO_{\widetilde{JC}}\bigl(\epsilon^*(2\Theta-E)\bigr)=\eta\boxtimes \eta,$$ cf. (\ref{eq:scorza_diagr}). Identifying
$H^0(JC, 2\Theta)_{0}$ with $H^0\bigl(\widetilde{JC}, \cO_{\widetilde{JC}}(2\Theta-E)\bigr)$, we  have a pull-back map
\begin{equation}\label{eq:ident3}
\varphi^* \colon H^0(JC, 2\Theta)_0\longrightarrow \mbox{Sym}^2 H^0(C, \omega_C).
\end{equation}
The map $\varphi^*$ assigns to a divisor $D\in |2\Theta|$ passing through $0\in JC$, its quadratic tangent cone at that point. With this preparation, we state our main result on the Scorza quartic.

\begin{theorem}\label{thm:scorza_quartic}
For a general even spin curve $[C, \eta]\in \cS_g^+$ of genus $g\geq 3$, the Scorza quartic $F(C,\eta)\colon \mathrm{Sym}^2 H^0(C,\omega_C)^{\vee}\longrightarrow  \mathrm{Sym}^2 H^0(C, \omega_C)$ can be realized via the following commutative diagram:

\xymatrix{
&  \mathrm{Sym}^2 H^0(C, \omega_C)^{\vee}\cong \mathrm{Sym}^2 H^0(\cO_{\Theta}(\Theta)) \ar[rr]^{}\ar[d]_{F(C, \eta)} &  & H^0\bigl(\Theta, \cO_{\Theta}(2\Theta)\bigr) \ar[d]^{\mathrm{res}^{-1}}&\\
&     \mathrm{Sym}^2 H^0(C, \omega_C) & & \ar[ll]_{\varphi^*}        H^0(JC, 2\Theta)_0}
\end{theorem}

The upper horizontal arrow in this diagram corresponds to multiplication of sections and uses (\ref{eq:ident2}). In order to establish that the map $F(C,\eta)\colon \mbox{Sym}^2 H^0(C, \omega_C)^{\vee}\longrightarrow \mbox{Sym}^2 H^0(C,\omega_C)$  lies in $\mbox{Sym}^4 H^0(C, \omega_C)$
and it gives rise to the Scorza quartic from \cite{DK}, we will make essential use of the \emph{Wirtinger duality} asserting the existence of an isomorphism $$\mathfrak{w}\colon |2\Theta|\stackrel{\cong}\lra |2\Theta|^{\vee}$$ having the defining property that for any point $a\in JC$, if $\Theta_a:=a+\Theta$, then
$$\mathfrak{w}\bigl(\Theta_a+\Theta_{-a}\bigr)=\bigl\{D\in \bigl|2\Theta\bigr|: a\in D\bigr\}.$$
In the course of proving Theorem \ref{thm:scorza_quartic}, we find an explicit description of $F(C, \eta)$ in terms of theta constants. Assuming
$$\theta_{\eta}=\theta_0+\theta_2+\theta_4+\cdots,$$
is the Taylor expansion at the origin of the symmetric theta function $\theta_{\eta}$ in terms of the canonical flat coordinates on $JC$, we have the following result:

\begin{theorem}\label{thm:scorza_quartic2}
For a general spin curve $[C, \eta]\in \cS_g^{+}$, we have that
$$F(C,\eta)=\frac{\theta_2^2}{2}-\theta_0\cdot \theta_4\in \mathrm{Sym}^4 H^0(C, \omega_C).$$
The limiting Scorza quartic for a general point $[C, \eta]\in \Theta_{\mathrm{null}}$ equals twice the quadratic tangent cone to the theta divisor $\Theta$ at the origin $0\in JC$.
\end{theorem}

\vskip 4pt

{\small{{\bf Acknowledgments:} We benefited from discussions with I. Dolgachev, R. Donagi, S. Grushevsky, A. Verra and F. Viviani related to this circle of ideas.

Farkas was supported by the Berlin Mathematics Research Center $\mathrm{MATH+}$ and by the ERC Advanced Grant SYZYGY. This project has received funding from the European Research Council (ERC) under the EU Horizon 2020  program (grant agreement No. 834172). Izadi was partially supported by the National Science Foundation.}}

\section{Scorza correspondences associated to spin structures}

We first collect basic facts that will be used throughout the paper.

\subsection{Moduli of spin and Prym curves.} We denote by  $\rr_g$ the moduli stack of stable Prym curves of genus $g$, see  \cite{Be, FL} for further background. Recall that an element of $\rr_g$ corresponds to a  triple $[X, \epsilon, \beta]$, where $X$ is a quasi-stable genus $g$ curve, $\epsilon$ is a locally free sheaf of total degree $0$ on $X$ such that $\epsilon_{|E}\cong \cO_E(1)$ for every smooth rational component $E\subseteq X$ with $|E\cap \overline{X\setminus E}|=2$ (such a component being called \emph{exceptional}) and $\beta\colon \epsilon^{\otimes 2}\rightarrow \cO_X$ is a sheaf morphism that is non-zero along each non-exceptional component of $X$. There exists a  finite map $\pi\colon \rr_g\rightarrow \mm_g$ assigning to a triple $[X, \eta, \beta]$  the stable model of the curve $X$. The ramification divisor
$\Delta_0^{\mathrm{ram}}$ of the map $\pi$ has as its general point a Prym curve  $[X, \epsilon, \beta]$, where $X:=C\cup_{\{x, y\}} E$ is a quasi-stable curve, with $C$ being a smooth curve of genus $g-1$ and $E$ being a smooth rational curve  meeting $C$ at two distinct points $x, y$ and  $\epsilon \in \mathrm{Pic}^0(X)$ satisfies $\epsilon_{|E}\cong \cO_{E}(1)$ and $\epsilon_{|C}^{\otimes 2}\cong \cO_C(-x-y)$.

\vskip 4pt

Let $\ss_g$ be the moduli space of stable spin curves of genus $g$. Depending on the parity of the spin structure, we distinguish two components $\ss_g^+$ and $\ss_g^-$ of $\ss_g$, whose geometry has been studied in detail in \cite{Cor, F2, FV}. We denote by $\mathfrak{M}_g$ the moduli space of smooth supersymmetric (susy) algebraic curves of genus $g$. It parametrizes smooth supervarieties of dimension $1|1$ endowed with a rank $0|1$ maximally non-integrable distribution of their tangent bundle. It is known \cite{DW, Wi, CV} that $\mathfrak{M}_g$ is a smooth Deligne-Mumford superstack of dimension $3g-3|2g-2$. The bosonic truncation of $\mathfrak{M}_g$ can be identified with the moduli stack $\cS_g$ of smooth spin curves of genus $g$ and we denote by $\mathfrak{M}_g^+$ the component of the susy moduli space such that $(\mathfrak{M}_g^+)_{\mathrm{bos}}=\cS_g^+$.

\vskip 3pt

\subsection{Tautological line bundles on symmetric squares of curves.} \label{subsecC_D} We set some notation. We fix a smooth algebraic curve $C$ of genus $g$ and denote  the two projections $\mathrm{pr}_i \colon  C\times C \lra C$, for $i=1,2$. We denote by $C^{(2)}$ the symmetric square of $C$ and let $q\colon C\times C\rightarrow C^{(2)}$ be the natural projection ramified along the diagonal $\Delta\subseteq C\times C$. Let $\delta\in \mbox{Pic}\bigl(C^{(2)}\bigr)$ be the line bundle uniquely characterized by the property $q^*(\delta)\cong \cO_{C^2}(\Delta)$. We also introduce the diagonal of the symmetric product
$$\overline{\Delta}=\bigl\{2\cdot p: p\in C\bigr\}\subseteq C^{(2)},$$
therefore $q^*(\overline{\Delta})=2\Delta$.

\vskip 4pt

For an effective divisor $D=p_1+\cdots+p_d$ on $C$ we introduce the divisor on the symmetric square
$$
C_D:=\bigl(C+p_1\bigr)+\bigl(C+p_2\bigr)+\cdots+\bigl(C+p_d\bigl)\in \mbox{Div}\bigl(C^{(2)}\bigr).
$$

This definition is extended to arbitrary line bundles $M=\cO_C(D_1-D_2)\in \mbox{Pic}(C)$, with $D_1$ and $D_2$ being effective divisors on $C$,  by setting
\begin{equation}\label{eq:pullback1}
\cL_{M}:=\cO_{C^{(2)}}\bigl(C_{D_1}-C_{D_2}\bigr)\in \mbox{Pic}\bigl(C^{(2)}\bigr).
\end{equation}

Note that $\cL_{M}$ can also be characterized as the unique line bundle on $C^{(2)}$ such that $q^*(\cL_M)\cong \mathrm{pr}_1^*(M)\otimes \mathrm{pr}_2^*(M)=M\boxtimes M$. Similarly, one defines $\cL_M':=\cL_M(-\delta)$ as the line bundle on $C^{(2)}$ such that
$q^*(\cL_M')\cong \mathrm{pr}_1^*(M)\otimes \mathrm{pr}_2^*(M)(-\Delta)$. The cohomology of these line bundles on $C^{(2)}$ is summarized for instance in \cite[Appendix]{I}:

\begin{align}\label{eq:cohsympr}
H^0(\cL_M)\cong \mbox{Sym}^2H^0(M), \ \ \ H^1(\cL_M)\cong H^1(M)\otimes H^0(M), \ \ \ H^2(\cL_M)\cong \bigwedge^2 H^1(M)\\
H^0(\cL_M')\cong \bigwedge^2 H^0(M), \ \ \ H^1(\cL_M')\cong H^1(M)\otimes H^0(M), \ \ \ H^2(\cL_M')\cong \mbox{Sym}^2 H^1(M).
\end{align}

\vskip 4pt

In the notation of Donagi and Witten \cite[3.3]{DW}, we have that for integers $a, c\in \mathbb Z$,
$$\cL_{\omega_C^{\otimes a}}(c\delta)\cong \cO_{C\times C}\bigl(a,a,c\bigr)^+.$$

\subsection{The Scorza correspondence and the Szeg\H{o} kernel}
We now fix an even theta-characteristic $\eta$ on a smooth curve $C$ of genus $g$ such that $h^0(C, \eta)=0$.  Following \cite{DK} we introduce the \emph{Scorza correspondence} associated to the pair
$[C, \eta]$, by setting
$$
S(C, \eta)=\Bigl\{(x,y)\in C\times C: h^0\bigl(C, \eta(x-y)\bigr) \neq 0\Bigr\}.
$$

The correspondence $S(C, \eta)$ being symmetric, it  is  the inverse under the map $q$ of a curve
$T(C, \eta)$ on $C^{(2)}$, that is,
$$S(C, \eta)=q^*\bigl(T(C, \eta)\bigr).$$
If $[C,\eta]\in \cS_g^+\setminus \Theta_{\mathrm{null}}$, then, since $S(C,\eta)$ is disjoint from the diagonal, it follows that the double cover $S(C,\eta)\rightarrow T(C,\eta)$ is unramified. It has been proved in \cite[Theorem 4.1]{FV} that, for a general  point $[C,\eta]\in \ss_g^+$, the curves $S(C,\eta)$ and $T(C,\eta)$ are smooth. In particular, by the adjunction formula, the genus of $S(C,\eta)$ equals $1+3g(g-1)$, whereas  the  genus of $T(C,\eta)$ equals $1+\frac{3}{2}g(g-1)$.

\vskip 4pt

The  morphism $\mathfrak{Sc}\colon \ss_g^+ \rightarrow \mm_{1+3g(g-1)}$ factors through the moduli space $\rr_{1+\frac{3}{2}g(g-1)}$ of Prym curves of genus $1+\frac{3}{2}g(g-1)$ and, following diagram (\ref{eq:diagram_stacks}), we can write
$$\mathfrak{Sc}=\chi \circ \xi,$$ where $\xi\bigl([C, \eta]):=\bigl[S(C, \eta)\stackrel{q}\longrightarrow T(C, \eta)\bigr]$ and  $\chi\colon \rr_{1+\frac{3}{2}g(g-1)}\rightarrow \mm_{1+3g(g-1)}$ is the morphism associating to an admissible double cover the stable model of its source. We refer to \cite[Proposition 4.1]{FL} for a detailed study of this map.

\vskip 3pt

We introduce the  difference map $$\varphi\colon C\times C \rightarrow JC, \ \ \ \  \varphi(x,y)=\cO_C(x-y).$$
Recalling that  $\Theta_{\eta}$ is  the symmetric \emph{Riemann theta divisor} associated to $\eta$ (see (\ref{eq:theta_eta})),  it is well-known, see  \cite[Proposition 7.1.5]{DK}, or \cite[3.4]{BZvB} that $\varphi^*\bigl(\cO(\Theta_{\eta})\bigr)\cong \eta\boxtimes \eta(\Delta)$, thus

\begin{equation}\label{eq:scorza_diagr}
\begin{tikzcd}[column sep=18pt]
\cO_{ C\times C } \bigl(S(C,\eta)\bigr) \cong \mathrm{pr}_1^* (\eta)\otimes \mathrm{pr}_2^* (\eta )\otimes \cO_{ C\times C} (\Delta)=\eta\boxtimes \eta(\Delta).
\end{tikzcd}
\end{equation}

\vskip 4pt

Using the canonical identification $\eta\boxtimes \eta(\Delta)_{\Delta} \cong \cO_C$, we can write down the following exact sequence on $C\times C$:
\[
0 \lra \eta\boxtimes \eta \lra \eta\boxtimes \eta(\Delta ) \lra \cO_C \lra 0,
\]
which gives rise to the following exact sequence of cohomology groups:
\begin{align*}
0 \lra H^0 (C, \eta )\otimes H^0 (C, \eta) \lra H^0 \bigl(C\times C, \eta\boxtimes \eta(\Delta)\bigr) \lra H^0 (C, \cO_C)\\ \lra \Bigl(H^0 (C, \eta )\otimes H^1 (C, \eta )\Bigr) \oplus \Bigl(H^1 (C,\eta )\otimes H^0 (C,\eta)\Bigr).
\end{align*}
Therefore, when $[C, \eta]\in \cS_g^+\setminus \Theta_{\mathrm{null}}$, that is, when $h^0 (C, \eta) =0$, we have that
\[
H^0 \bigl(C\times C, \eta\boxtimes \eta(\Delta)\bigr) \cong H^0 (C, \cO_C).
\]
From \eqref{eq:scorza_diagr}, we obtain that
\begin{equation}\label{eqRlin}
\cO_{C^{(2)}}\bigl(T(C,\eta)\bigr)\cong \cL_{\eta}(\delta).
\end{equation}

\subsection{From the Scorza curve to the Szeg\H{o} kernel}

We now explain how the Scorza curve $S(C, \eta)$  can be viewed as the zero locus of the Szeg\H{o} kernel defined by (\ref{eq:szego}).

\begin{definition}\label{def:szego} The (symmetric) \emph{Szeg\H{o} kernel} of the pair $[C, \eta]\in \cS_g^+\setminus \Theta_{\mathrm{null}}$ is defined as the unique section $s_{\eta}$ of the line bundle $\eta\boxtimes \eta(\Delta)$ on $C\times C$ such that
such that $s_{\eta|\Delta}=\mbox{id}_{\Delta}$.
\end{definition}

By definition, $\mathrm{s}_{\eta}$ is a symmetric function on $C\times C$ and the Scorza scurve $S(C,\eta)$ is its vanishing locus.

\vskip 3pt

If $h^0(C,\eta)=0$, then $S(C, \eta)$ is a uniquely determined curve in its linear system on $C\times C$,
whereas, when $h^0 (C, \eta) =2$, we have
\[
 h^0 \bigl(C\times C, \eta\boxtimes \eta(\Delta)\bigr) \geq 4.
\]
The limit of the Scorza correspondence for a general point $[C,\eta]\in \Theta_{\mathrm{null}}$, viewed as a curve in $C\times C$, belongs to the linear system $\bigl|\eta\boxtimes \eta(\Delta)\bigr|$ which has  dimension at least $3$.  This limit curve is  symmetric, that is, the inverse image of a divisor on $C^{(2)}$ which has to be an element of the linear system $\bigl|\cL_{\eta}(\delta)\bigr|$ on $C^{(2)}$. The following calculation shows that the limit of the Scorza correspondence when $h^0 (C, \eta )=2$ exists and it is unique.

\begin{proposition}\label{lemeta02}
If $[C, \eta]\in \cS_g^+$ is such that $h^0(C, \eta)\leq 2$, then $h^0\bigl(C^{(2)}, \cL_{\eta}(\delta)\bigr) = 1$.
\end{proposition}

\begin{proof}
We consider the following exact sequence on $C^{(2)}$
$$0\longrightarrow \cL_{\eta}(-2\delta)\longrightarrow \cL_{\eta}\longrightarrow \cO_{\overline{\Delta}}(\cL_{\eta})\longrightarrow 0,$$
noting that $\cO_{\overline{\Delta}}(\cL_{\eta})\cong \eta^{\otimes 2}\cong \omega_C$. Writing the long exact sequence in cohomology and using the identifications (\ref{eq:cohsympr}) we obtain the following exact sequence:

\begin{align}
0 \lra H^0\bigl(C^{(2)}, \cL_{\eta }(- 2\delta)\bigr) \lra \Sym^2 H^0 (\eta ) \stackrel{\cup_0}{\lra} H^0 ( \omega_C ) \lra H^1\bigl(C^{(2)}, \cL_{\eta }(- 2\delta)\bigr)\\
\lra H^0 (\eta ) \otimes H^1 (\eta )
\stackrel{\cup_1}{\lra} H^1 (\omega_C ) \lra H^2\bigl(C^{(2)}, \cL_{\eta }(-2\delta)\bigr) \lra \bigwedge^2 H^1 (\eta ) \lra 0,
\end{align}
where the maps $\cup_0$ and $\cup_1$ are given by cup-product. By, e.g., \cite[Appendix]{I}, we have that $\omega_{C^{(2)}}\cong \cL_{\omega_C}(-\delta)$, therefore by Serre duality we also obtain that
$$H^2\bigl(C^{(2)}, \cL_{\eta}(-2\delta)\bigr)\cong H^0\bigl(C^{(2)}, \cL_{\omega_C}\otimes \cL_{\eta}^{\vee}(-\delta+2\delta)\bigr)^{\vee} \cong H^0\bigl(C^{(2)}, \cL_{\eta}(\delta)\bigr)^{\vee}.$$
When $h^0(C,\eta)=0$, then clearly $H^0\bigl(C^{(2)}, \cL_{\eta}(\delta)\bigr)^{\vee}\cong H^1(C,\omega_C)\cong \mathbb C$. When, on the other hand, $h^0(C, \eta)=2$, we then observe that the cup product map $\cup_1$ is surjective, because it is nonzero and $h^1(C,\omega_C)=1$. Therefore we obtain the canonical identification
$$H^0\bigl(C^{(2)}, \cL_{\eta}(\delta)\bigr)^{\vee}\cong \bigwedge^2 H^1(C,\eta)\cong \mathbb C.$$
We conclude that whenever $h^0(C,\eta)\leq 2$, the linear system $|\cL_{\eta}(\delta)|$ contains a unique curve, which is the limit in $C^{(2)}$ of the Scorza correspondence.
\end{proof}

\subsection{The limit of the Scorza correspondence on $\cS_g^+$}

The conclusion of Proposition \ref{lemeta02} offers a way to put forward a definition of the curves $T(C, \eta)$ and $S(C,\eta)$ for every smooth curve $[C, \eta]\in \cS_g^+$. However, this potentially leads to non-stable (or even non-reduced) curves in $C^{(2)}$, respectively, in $C\times C$.

\vskip 3pt

Let $\rho \colon \cC_g \ra \cS_g^+$ be the universal smooth even spin curve of genus $g$, that is, the stack of triples $[C, \eta, p]$, where $[C, \eta]\in \cS_g^+$ and $p\in C$. We denote by $\cC^2_g \ra \cS_g^+$, respectively, by $\rho_2 \colon \cC^{(2)}_g \ra \cS_g^+$ the relative Cartesian and symmetric powers of the universal curve $\cC_g \ra \cS_g^+$. The symmetrization morphism $q\colon \cC^2_g \ra \cC^{(2)}_g$ is ramified along the relative diagonal $\cD_g \subseteq \cC^2_g$ and we let $\delta_g$ be the divisor class on $\cC^{(2)}_g$ whose pull-back under $q$ is $\cD_g$. Let $\eta_g$ be a universal even spin bundle\footnote{Note that since $\cS_g$ has the structure of a $\mathbb Z_2$-gerbe due to the fact that each stable spin curve has a copy of $\mathbb Z_2$ in ts automorphism group, the universal line bundle $\eta_g$ should be viewed as a $\mathbb Z_2$-twisted line bundle over $\mathcal{C}_g$, see also the discussion in \cite[1.1]{DW}.} over $\cC_g$. By definition, the restriction of $\eta_g$ to the fiber of $\cC_g \ra \cS_g^+$ over  $[C, \eta]$ is the spin bundle $\eta\in \mbox{Pic}^{g-1}(C)$.
We denote $\cL_{\eta_g}$ the line bundle on $\cC^{(2)}_g$ such that
\begin{equation}\label{eq:univ_spin3}
q^*(\cL_{\eta_g})=\mathrm{pr}_1^*(\eta_g) \otimes \mathrm{pr}_2^*(\eta_g),
\end{equation}
where $\mathrm{pr}_i\colon \cC^{2}_g\rightarrow \cC_g$ are the two projections.

\vskip 4pt

By Proposition \ref{lemeta02}, the generic fiber of $({\rho_2})_*\bigl(\cL_{\eta_g}(\delta_g)\bigr)$ is one-dimensional. Hence, since $\rho_2$ is flat, the push-forward $({\rho_2})_*\bigl(\cL_{\eta_g}(\delta_g)\bigr)$ is a reflexive sheaf of rank one. Via the Auslander-Buchsbaum formula, there exists an open subset $\cS_g^{\mathrm{free}}$ of $\cS_g^+$ whose complement $\cS_g^+\setminus \cS_g^{\mathrm{free}}$ has codimension at least $3$ such that $(\rho_2)_*\bigl(\cL_{\eta_g}(\delta_g)\bigr)$ is locally free over $\cS_g^{\mathrm{free}}$.

Clearly, $\eta_g$ is only determined up to the pull-back under $\rho$ of a line bundle from $\cS_g^+$ and we normalize $\eta_g$ in such a way that
$$(\rho_2)_*\bigl(\cL_{\eta_g}(\delta_g)\bigr) |_{\cS_g^{\mathrm{free}}}\cong \cO_{\cS_g^{\mathrm{free}}}.$$

\begin{definition}\label{def:limitscorza}
For any smooth spin curve $[C ,\eta]\in \cS_g^{\mathrm{free}}$, we define the limiting Scorza correspondence inside the symmetric product $C^{(2)}$ to be the image of  the natural map
\[
({\rho_2})_*\left(\cL_{\eta_g}(\delta_g)\right)_{\bigl|[C ,\eta]} \lra H^0\bigl(C^{(2)} , \cL_{\eta}(\delta)\bigr).
\]
\end{definition}

Definition \ref{def:limitscorza} singles out a one-dimensional subspace of $H^0\bigl(C^{(2)} , \cL_{\eta}(\delta)\bigr)$ for a spin curve $[C ,\eta]\in \cS_g^{\mathrm{free}}$. The inverse image in $C\times C$ of the zero scheme of a non-zero element of this one-dimensional subspace gives a well-defined, possibly non-reduced, curve in $C\times C$.

\subsection{The limit of the Scorza correspondence at a general point of $\Theta_{\mathrm{null}}$} In what follows,  we determine the limit of the Scorza correspondence when $h^0 (C, \eta) =2$. To that end, we fix a general point $[C, \eta]\in \Theta_{\mathrm{null}}$, therefore $|\eta|$ induces a cover $f\colon C\rightarrow \PP^1$ of degree $g-1$. We may assume that the pencil $|\eta|\in W^1_{g-1}(C)$ has only simple ramification points. The ramification divisor $\mathfrak{Ram}$ of $f$ belongs to the linear system $|\omega_C\otimes f^*(\cO_{\PP^1}(2))|=\bigl|\omega_C^{\otimes 2}\bigr|$ and consists of $4g-4$ distinct points. We write $\mathfrak{Ram}=x_1+\cdots+x_{4g-4}$. In particular, the canonical bundle $\omega_C$ induces canonically a double cover
\begin{equation}\label{eq:candouble}
\widetilde{C}_{\eta}\longrightarrow C
\end{equation}
branched over the points $x_1, \ldots, x_{4g-4}$. Note that $g(\widetilde{C}_{\eta})=4g-3$.

\vskip 3pt

We introduce the \emph{trace curve}
\begin{equation}\label{eq:trace_curve}
\Gamma_{\eta}:=\Bigl\{x+y\in C^{(2)}: H^0\bigl(C, \eta(-x-y)\bigr)\neq 0\Bigr\}
\end{equation}
and its double cover $\widetilde{\Gamma}_{\eta}:=q^{-1}\bigl(\Gamma_{\eta}\bigr)\subseteq C\times C$. It is well known, see, e.g., \cite[Lemma 2.1]{I} that
\begin{equation}\label{eq:x2}
\cO_{C^{(2)}}\bigl(\Gamma_{\eta}\bigr)\cong  \cL_{\eta}(-\delta).
\end{equation}
By the adjunction formula, using again that $\omega_{C^{(2)}}\cong \cL_{\omega_C}(-\delta)$, we find that $\Gamma_{\eta}$ is a curve of arithmetic genus $g\bigl(\Gamma_{\eta}\bigr)=\frac{(g-3)(3g-4)}{2}$. Since $|\eta|$ is simply ramified, using, for instance, \cite[Lemmas 2.1 and 2.2]{vdGK2}, we conclude that, for a general choice of $[C, \eta]\in \Theta_{\mathrm{null}}$, the curves $\Gamma_{\eta}$ and $\widetilde{\Gamma}_{\eta}$ are smooth and irreducible.

\vskip 4pt

\begin{theorem}\label{propthetanull2limit} We fix a general vanishing theta-null $[C, \eta]\in \Theta_{\mathrm{null}}$.

\vskip 3pt

\noindent (1) The limiting Scorza correspondence $T(C,\eta)$  is the transverse union of $\Gamma_{\eta}$ and the diagonal $\widetilde{\Delta}\cong C$, meeting at the $4g-4$ diagonal points $2\cdot x_1, \ldots, 2\cdot x_{4g-4}\in C^{(2)}$.

\vskip 5pt

\noindent (2) The limiting Scorza correspondence $S(C, \eta)$ is the transverse union of $\widetilde{\Gamma}_{\eta}$ and the curve $\widetilde{C}_{\eta}$ meeting at the $(4g-4)$  diagonal points  $(x_1, x_1), \ldots, (x_{4g-4}, x_{4g-4})$.
\end{theorem}

\begin{proof}
We pick a general point $[C,\eta]\in \thet$. Since $\Gamma_{\eta}\in \bigl|\cL_{\eta}(-\delta)\bigr|$, whereas $\overline{\Delta}\in |2\delta|$, it follows that that $\Gamma_{\eta}+\overline{\Delta}\in \bigl|\cL_{\eta}(\delta)\bigr|$. Since in Lemma \ref{lemeta02} we also established that $h^0\bigl(C^{(2)}, \cL_{\eta}(\delta)\bigr)=1$, it follows that $\Gamma_{\eta}+\overline{\Delta}$ is the only curve in the linear system $\bigl|\cL_{\eta}(\eta)\bigr|$. Note that $\Gamma_{\eta}$ and $\overline{\Delta}$ intersect at the $4g-4$ point of the form $2\cdot x$, where $x\in \mathfrak{Ram}$. Since
$$g\bigl(\Gamma_{\eta}\bigr)+g(\overline{\Delta})+\#\bigl(\Gamma_{\eta}\cap \overline{\Delta}\bigr)-1=\frac{(g-3)(3g-4)}{2}+g-1+4(g-1)=1+\frac{3g(g-1)}{2},$$
it also follows that the intersection of the curves $\Gamma_{\eta}$ and $\overline{\Delta}\cong C$ is everywhere transverse, in particular the  curve $\Gamma_{\eta}\cup \overline{\Delta}$ is stable and is therefore the limiting Scorza correspondence $T(C,\eta)$.

\vskip 4pt

Passing to the cartesian product, the limiting Scorza correspondence as a curve inside $C\times C$ is then the \emph{non-reduced} curve $2\Delta+\widetilde{\Gamma}_{\eta}$. The corresponding stable curve $S(C,\eta)$ must be an admissible double cover having as target a nodal curve stably equivalent to the stable curve $\Gamma_{\eta} \cup \overline{\Delta}$, where we identify $\overline{\Delta}$ with $C$. For
$i=1, \ldots, 4g-4$, we insert a smooth rational curve $E_i\cong \PP^1$ meeting $\Gamma_{\eta}$ at the point $\tilde{x}_i:=2\cdot x_i$ (viewed as a point of $\Gamma_{\eta}$) and $\overline{\Delta}$ at the point $x_i^-:=2\cdot x_i$ (viewed as a point of $\overline{\Delta}$). We then consider the following admissible double cover
\begin{equation}\label{eq:double_cov}
f\colon \widetilde{\Gamma}_{\eta}\cup \widetilde{E}_1\cup \ldots \cup \widetilde{E}_{4g-4}\cup \widetilde{C}_\eta \longrightarrow \Gamma_{\eta}\cup E_1\cup \ldots \cup E_{4g-4}\cup C,
\end{equation}
where $f_{|\widetilde{C}_{\eta}}\colon \widetilde{C}_{\eta}\rightarrow C$ is the cover described by (\ref{eq:candouble}), $f_{|\widetilde{\Gamma}_{\eta}}\colon \widetilde{\Gamma}_{\eta}\rightarrow \Gamma_{\eta}$ is the cover induced by the map $q$, whereas for $i=1, \ldots, 4g-4$, the restriction $f_{|\widetilde{E}_i}\colon \widetilde{E}_i\rightarrow E_i$ is the double cover branched over the points $\tilde{x}_i$ and $x_i^-$.

\vskip 3pt

Note that this is the procedure prescribed by the Stable Reduction Theorem, see , e.g., \cite[p. 125]{HMo} in order to eliminate the multiple component in the limiting curve $2\Delta+\widetilde{\Gamma}_{\eta}$. Indeed, assuming we have a family of curves $\varphi\colon \cX\rightarrow (T, t_0)$, where $T$ is a smooth $1$-dimensional base such that $\varphi^{-1}(t)$ is a smooth curve for $t\in T\setminus\{t_0\}$, whereas $\varphi^{*}(0)=2\Delta+\widetilde{\Gamma}_{\eta}$, then the total space $\cX$ is necessarily singular at the points $\Delta\cdot \widetilde{\Gamma}_{\eta}$. Blowing these $4g-4$ points up,  making a base change $(T, t_0)\rightarrow (T, t_0)$ of order $2$ and then normalizing, the resulting fibration $\varphi'\colon \mathcal{X}'\rightarrow T$ has central fibre equal to $\widetilde{\Gamma}_{\eta}+E_1'+\cdots+E_{4g-4}'+\widetilde{C}_{\eta}$, where $E_i'$ is the inverse image of the corresponding exceptional divisor of $\mathrm{Bl}_{4g-4}(\cX)$, whereas $\widetilde{C}_{\eta}$ is the double cover of $\Delta\cong C$ branched over the points  $\Delta\cdot E_i'$. This finishes the proof.
\end{proof}

\section{Extension of the Scorza map to the boundary of $\ss_g^+$}\label{sect:scorza_bdry}

 We now turn our attention to the limits of the Scorza curves on the boundary of the moduli space $\ss_g^+$ of stable even spin curves of genus $g$. One has the rational map
$$\mathfrak{Sc}\colon \ss_g^+\dashrightarrow \mm_{1+3g(g-1)}, \ \ \ \mathfrak{Sc}([C, \eta]):=[S(C, \eta)],$$
considered in the Introduction. Since $\ss_g^+$ is a normal variety, $\mathfrak{Sc}$ is defined outside a set of codimension at least $2$, thus the limit Scorza curves exist generically, as stable curves, on each boundary component of $\ss_g^+$. We now describe explicitly a modular extension of $\mathfrak{Sc}$, generically on each boundary component.

\vskip 4pt

For basic facts about the boundary divisors of $\ss_g^+$ we refer to \cite{Cor, F2}. The boundary divisors of $\ss_g^+$ are traditionally denoted by $A_i, B_i$, where $i=0, \ldots, \lfloor \frac{g}{2}\rfloor$.


\subsection{The limit Scorza correspondence over a general point of $A_0$.}

We begin with a general point of the boundary divisor  $A_0$ of $\ss_g^+$, corresponding to the following data:
\begin{itemize}
\item A general $2$-pointed curve $[C, p, q]\in \cM_{g-1, 2}$. We denote by $X:=C/p\sim q$ the stable curve of genus $g$ obtained by identifying $p$ and $q$ and let $\nu\colon C\rightarrow X$ be the normalization map with $\nu(p)=\nu(q)=u\in X$.
 \item A line bundle $\eta_C\in \mathrm{Pic}^{g-1}(C)$ such that $\eta_C^{\otimes 2}\cong \omega_C(p+q)$ and $H^0(X, \eta)=0$, where we denote by $\eta\in \mathrm{Pic}(X)$ the locally free sheaf such that $\nu^*(\eta)=\eta_C$.
 \end{itemize}
Note that giving $\eta\in \mathrm{Pic}^{g-1}(X)$ is equivalent to specifying $\eta_C$, as well as a gluing between the fibres $\eta_{_C|p}$ and $\eta_{_C|q}$.

\vskip 4pt

Let $\nu\times \nu \colon C\times C\rightarrow X\times X$ be the product map. The limit Scorza correspondence is a symmetric curve in $X\times X$ and we let $\Sigma$ be its pull-back under $\nu\times \nu$.
 We have the following exact sequence on $X$:
\[
0\longrightarrow \cO_X\longrightarrow \nu_*\cO_C\longrightarrow \mathbb C_u\longrightarrow 0
\]
which, for any $x,y \in C\setminus \{p, q\}$ gives us the exact sequence
\begin{equation}\label{eq:exseqA0}
0\longrightarrow \eta (x-y) \longrightarrow \nu_*\bigl(\eta_C (x-y)\bigr) \longrightarrow \mathbb C_u\longrightarrow 0,
\end{equation}
where to simplify notation, we identify $\nu(x)$ with $x$ and $\nu(y)$ with $y$. We denote by

 $$\sigma_{x, y}\in H^0\bigl(C, \eta_C(x-y)\bigr)$$ a non-zero section. The limiting Scorza curve $S(X, \eta)$ corresponding to the point $[X, \eta]\in A_0$ contains the closure of the set of pairs $(x,y)\in X\times X$ such that $h^0(X, \eta(x - y)) \neq 0$, that is, the image of $\sigma_{x,y}$ is zero in $\bC_u$ in the exact sequence (\ref{eq:exseqA0}). Note that $\sigma_{x, y}\cdot \sigma_{y, x}\in H^0(C, \eta_C^{\otimes 2})=H^0(C, \omega_C(p+q))$ is a meromorphic differential
 on $C$ with non-zero residue at both $p$ and $q$, unique modulo the image of $H^0(C, \omega_C)$.

\begin{definition}
Given a general spin curve $[X=C/p\sim q, \eta]\in A_0$ as above, we define the correspondence $\Sigma$ as the closure in $C\times C$ of the locus

 \begin{align*}
 \Sigma_0:=\Bigl\{(x, y)\in \bigl(C \setminus \{p, q\}\bigr) \times \bigl(C \setminus \{p, q\}\bigr) :\\
  \sigma_{x, y}(p)=\sigma_{x,y}(q),\  \mbox{ for } \ \ \ 0\neq \sigma_{x,y}\in H^0\bigl(C, \eta_C(x-y)\bigr)\Bigr\}.
\end{align*}

\end{definition}

The equality $\sigma_{x,y}(p)=\sigma_{x,y}(q)$ is to be understood in terms of the identification of the fibres $\eta_{C}(x-y)_{|p}$ and $\eta_C(x-y)_{|q}$ which is part of the data defining $\eta$. Next we determine the class of the curve $\Sigma$:

 \begin{proposition}\label{prop:class_sigma}
The correspondence $\Sigma$ is symmetric of valence $g$. Furthermore $\Sigma$ intersects the diagonal $\Delta$ at the points $(p, p)$ and  $(q, q)$.
 \end{proposition}
 \begin{proof}
The symmetry follows from the fact that the image of $\Sigma$ in $X\times X$ is symmetric.


Since $C$ is general, we can write $\Sigma\equiv a(F_1+F_2)+b\Delta$. To calculate the intersection of $\Sigma$ with a general fibre $F_1$ of the projection $C\times C\rightarrow C$, we choose a point $x\in C \setminus \{p, q\}$. From Riemann-Roch it follows that $H^0(X, \eta(x))$ is one-dimensional. Let $\sigma_x$ be a generator. The points $y\in C$ such that $(x, y)\in \Sigma$ are the zeroes of $\sigma_x$ which, via the Mayer-Vietoris sequence, can be considered a nonzero section of a line bundle of degree $g=\mbox{deg}(\eta_C)+1$ on $C$, that is, $\Sigma \cdot F_1=g$.

\vskip 3pt

Assume now that $(x,x) \in S(X, \eta)$. If $x\in C\setminus \{p,q\}$, then, by semicontinuity, we obtain that $h^0(X, \eta)\geq 1$, which is a contradiction. To deal with the case $x=p$, we construct a family $\varphi\colon \mathcal{X}\rightarrow (B,0)$ of spin curves of genus $g$ endowed with two section $\bar{x}, \bar{y}\colon B\rightarrow \mathcal{X}$, such that $\varphi^{-1}(0)=C\cup E\cup E'$, with $E$ and $E'$ smooth rational curves such that $C\cap E'=\{p\}$, $C\cap E=\{q\}$ and $E\cap E'$ is one point that we denote $p'$. Furthermore,
$$\bar{x}(0)=x\in E'\setminus \{p, p'\}\ \ \mbox{ and } \ \ \bar{y}(0)=y\in E'\setminus \{p,p'\}$$ are distinct points on $E'$. We may also assume that each fibre $X_b=\varphi^{-1}(b)$ is endowed with a spin structure $\eta_b\in \mbox{Pic}^{g-1}(X_b)$ such that $h^0\bigl(X_b, \eta_b(\bar{x}(b)-\bar{y}(b))\bigr)\geq 1$. Since the limiting spin structure on the central fibre $\varphi^{-1}(0)=C\cup E$ is given by that corresponding to the curve $[X, \eta]\in A_0$ we started with, it follows that $\eta_{0|C}=\eta_C$, $\eta_{0|E'}=\cO_{E'}$ and $\eta_{0|E}=\cO_E$. Writing down the Mayer-Vietoris sequence on $C\cup E'\cup E$, we obtain:
$$0\longrightarrow H^0\bigl(C\cup E\cup E', \eta_0(x-y)\bigr)\longrightarrow H^0(C, \eta_C)\oplus H^0\bigl(E', \cO_{E'}(x-y)\bigr)\oplus H^0(E, \cO_E) \stackrel{\mathrm{ev}_{p,q}}\longrightarrow \mathbb C^3_{p,p',q}.$$ Denoting by $\sigma_C$ a generator of the vector space $H^0(C,\eta_C)$, it follows that, given $x$, the point $y\in E$ is uniquely determined by the condition that the unique section $\sigma_{E'}\in H^0(E', \cO_{E'}(x))$ that is compatible with $\sigma_C$ and with a nowhere vanishing generator $\sigma_E\in H^0(E,\cO_E)$ vanishes at the point $y$. This shows that $(p,p)\in \Sigma$ and the same argument yields that $(q,q)\in \Sigma$.
\end{proof}

\begin{proposition}
The genus of $\Sigma$ equals $g(\Sigma)=(g-1)(3g-2)$.
\end{proposition}
\begin{proof} Use Proposition \ref{prop:class_sigma} to write $\Sigma\equiv (g-1)(F_1+F_2)+\Delta$ and  couple this with the adjunction formula.
\end{proof}

Note that $h^0(C, \eta_C)=1$ and write $\mbox{supp}(\eta_C)=\{x_1, \ldots, x_{g-1}\}$. We may assume that $p,q\notin \mbox{supp}(\eta_C)$. By Serre Duality, we find that $H^0(C, \eta_C(x_i-p-q))\neq 0$,  therefore $(x_i, p), (x_i, q)\in \Sigma$ for $i=1, \ldots, g-1$. By symmetry, we also have that $(p, x_i), (q, x_i)\in \Sigma$.

\vskip 4pt

We are in a position to describe the limiting Scorza curve for a general point of $A_0$.

\begin{proposition}\label{prop:limitA0}
The limit Scorza curve $S(X,\eta)$ corresponding to a general spin curve $[X,\eta]\in A_0$ is the irreducible stable curve obtained from the smooth curve $\Sigma\subseteq C\times C$ by identifying the following $2g-1$ pairs of points:
\[
\begin{tikzcd}[column sep=18pt]
(p, p)\sim (q, q) \ \mathrm{ and } \ (x_i, p)\sim (x_i, q), \ \ \ (p, x_i)\sim (q, x_i) \ \mathrm{ for } \ i=1, \ldots, g-1.
\end{tikzcd}
\]
\end{proposition}

Note that the genus of the resulting curve is, as it should be,
$$p_a \bigl(S(X,\eta)\bigr)=g(\Sigma)+2(g-1)+1=1+3g(g-1).$$
\vskip 3pt
\begin{proof}
We already explained that $(\nu\times \nu)^{-1}\bigl(S(X,\eta)\bigr)=\Sigma$, therefore $S(X,\eta)$ is obtained from $\Sigma$ by possibly identifying pairs of points of the form $(x,p)$ and $(x,q)$ (respectively $(p,x)$ and $(q,x)$). Since, by Proposition \ref{prop:class_sigma}, the valence of $\Sigma$ equals $g$ and we have already exhibited $g$ distinct points lying in the fibre of $\Sigma$ over the point $p\in C$ (respectively $q$), we find that $\Sigma \cdot \mathrm{pr}_1^{-1}(p)=(p,p)+(p,x_1)+\cdots+(p,x_{g-1})$ and $\Sigma \cdot \mathrm{pr}_1^{-1}(q) =(q,q)+(q,x_1)+\cdots+(q,x_{g-1})$, respectively. A similar statement holds with respect to the projection $\mathrm{pr}_2$. It thus follows that these pairs of points get identified, which brings the proof to an end.
\end{proof}

\subsection{The limit Scorza curve for a general point of $B_0$.}

We begin with a general element $[X, \eta]$ of $B_0$, corresponding to a nodal curve $X:=C\cup_{\{p, q\}} E$, where $E$ is a smooth rational curve meeting $C$ at $p$ and $q$, as well as to a line bundle $\eta\in \mbox{Pic}^{g-1}(X)$, whose restrictions to the components of $X$ are an even theta characteristic $\eta_C := \eta_{| C}$ on $C$, respectively, $\eta_{|E}=\cO_E(1)$. Note that $[C, p,q]\in \cM_{g-1,2}$ may be assumed to be general.
We consider the decomposition of the cartesian product
\[
X \times X = C\times C \ \cup \ C\times E\ \cup \ E\times C\ \cup E\ \times E,
\]
where the points $p$ and $q$ are thought of as lying on both $C$ and $E$.
In $C\times C$ one considers  the Scorza curve $S(C, \eta)$ which is smooth of genus $1+3(g-1)(g-2)$. By semicontinuity, $S(C, \eta)$ appears as a component of the limit Scorza curve. We then consider the pencil
$|\eta_C(p+q)| \in W^1_g(C)$ inducing a regular map $f\colon C\rightarrow E = \PP^1$ and an embedding
$$(1, f)\colon C\longrightarrow C\times E$$ whose image we denote by $C_1$.
Permuting the factors, we define $C_2\subseteq E\times C$ to be the image of the map $(f, 1)\colon C\rightarrow E\times C$.

\vskip 4pt

We set $S(C, \eta_C)_p=\{x_1, \ldots, x_{g-1}\}$ and $S(C, \eta_C)_q=\{x_1', \ldots, x_{g-1}'\}$. We normalize the map $f$ in such a way that
\begin{equation}\label{eq:identifB0}
f(x_1)=\cdots =f(x_{g-1})=f(q)=p\in E \ \mbox{ and } \ f(x_1')=\cdots=f(x_{g-1}')=f(p)=q\in E.
\end{equation}

We identify the points $(x_i, p)\in S(C, \eta_C)$ and $(x_i, f(x_i))=(x_i, p)\in C_1$, as well as the points $(x_i', q)\in S(C,\eta_C)$ and,
respectively, $(x_i',f(x_i'))=(x_i',q)\in C_1$ for $i=1, \ldots, g-1$, where we use (\ref{eq:identifB0}) throughout.  Similarly, we identify the points $(p, x_i)\in S(C, \eta_C)$ and $(f(x_i),x_i)=(p,x_i)\in C_2$, as well as the pairs of points $(q,x_i')\in S(C,\eta_C)$ and $(f(x_i'),x_i')=(q,x_i')\in C_2$ for $i=1,\ldots, g-1$, respectively. Therefore $S(C,\eta_C)$ meets transversally both $C_1$ and $C_2$ at $2g-2$ points. Finally the points
\begin{align}\label{eq:identifB02}
(f(q),q)=(p,q)\in C_2,\ \  \  (f(p),p)=(q,p)\in C_2\\
(p,f(p))=(p,q)\in C_1, \ \  \ (q,f(q))=(q,p)\in C_1
\end{align}
can also be identified in pairs, thus $C_1$ and $C_2$ meet transversally in two points denoted by $(p,q)$, respectively, $(q,p)$.

\begin{figure}[!h]
\begin{center}
\hspace{2cm}
\includegraphics[width=8cm, height=6cm]{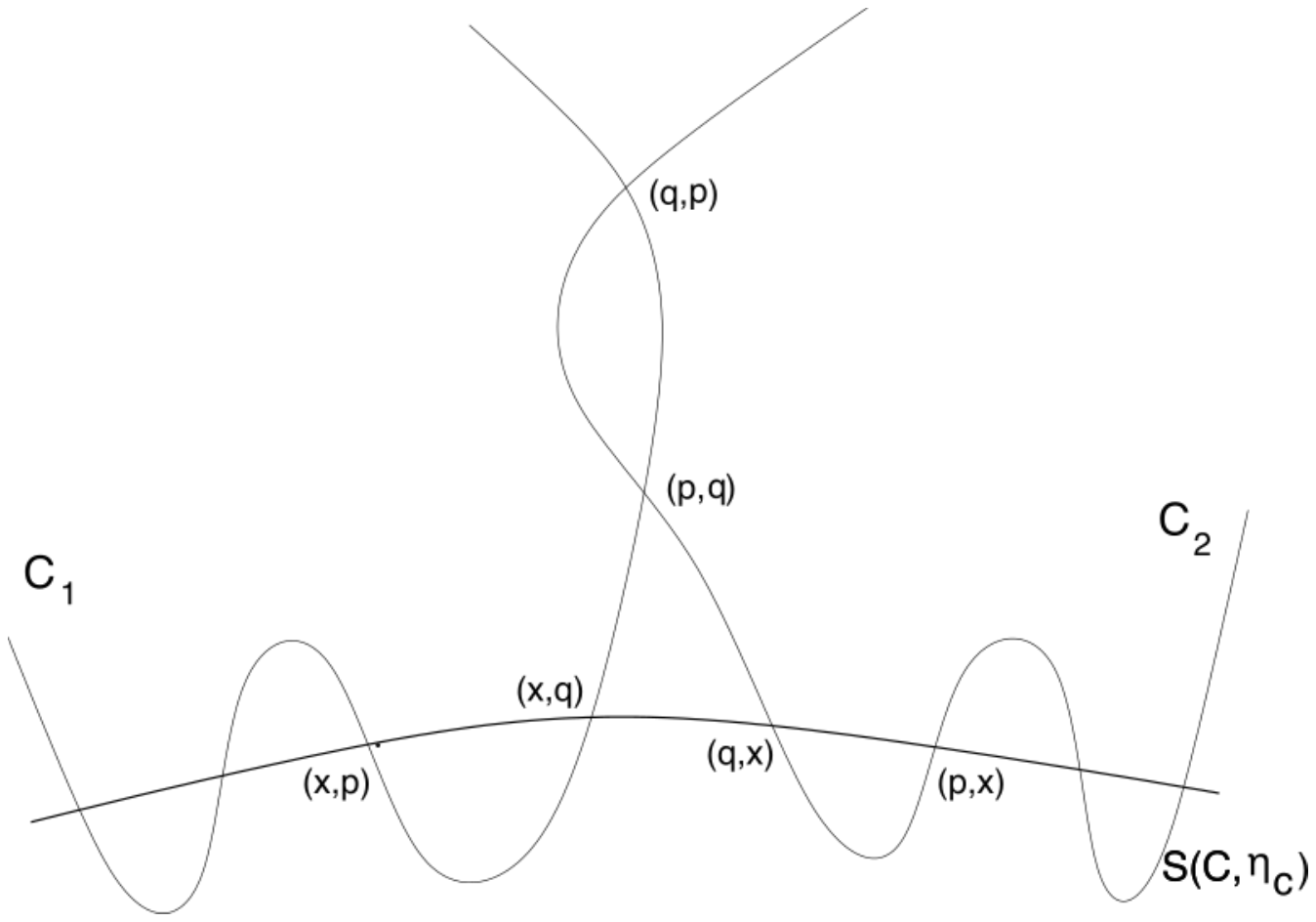}
\caption{The Scorza curve corresponding to a general point of $B_0$}
\end{center}
\end{figure}

%

\begin{proposition}\label{prop:B0}
The Scorza curve corresponding to a general point of $B_0$ is the stable curve
$$S(C, \eta_C) \ \cup C_1 \ \cup \ C_2,$$
where its components meet along the identifications described by (\ref{eq:identifB0}) and (\ref{eq:identifB02}).
\end{proposition}

\begin{proof}
Let $\bigl(\varphi \colon \cX\rightarrow (B, 0), \eta_{\cX}\bigr)$ be a one-dimensional family of stable spin curves with
$$
\begin{tikzcd}[column sep=28pt]
\bigl(\varphi^{-1}(0), \ \eta_{\cX| \varphi^{-1}(0)}\bigr)=\bigl(C\cup_{\{p,q\}} E, \eta \bigr)
\end{tikzcd}$$ as before, and such that $\bar{x}, \bar{y}\colon B\rightarrow \cX$ are sections of $\varphi$ with $h^0\bigl(X_b, \eta_b(\bar{x}(b)-\bar{y}(b))\bigr)\geq 1$ for each $b\in B$, where $\eta_b:=\eta_{| X_b}$. Set $\overline{x}(0)=:x$ and $\overline{y}(0)=:y$.

\vskip 3pt

If $x, y\in C\setminus \{p,q\}$, then using the isomorphism $H^0\bigl(X, \eta(x-y)\bigr)\cong H^0\bigl(C, \eta_C(x-y)\bigr)$, we obtain that $(x,y)\in S(C,\eta_C)$, which accounts for this component in Proposition \ref{prop:B0}. Let us assume now that $x\in C\setminus \{p,q\}$ and that $y\in E\setminus \{p,q\}$. Writing down the Mayer-Vietoris sequence on $X$ we obtain:

$$0\lra H^0\bigl(X, \eta(x-y)\bigr)\lra H^0\bigl(C, \eta_C(x)\bigr)\oplus H^0\bigl(E, \cO_E(1)(-y)\bigr)\stackrel{\mathrm{ev}_{p,q}}\lra \mathbb C^2_{p,q}.$$
Therefore there exists a section $0\neq \sigma_C\in H^0\bigl(C, \eta_C(x)\bigr)$ and a section $\sigma_E\in H^0\bigl(E, \cO_E(1)\bigr)$ uniquely determined by the conditions $\sigma_E(p)=\sigma_C(p)$ and $\sigma_E(q)=\sigma_C(q)$. The point $y\in E$ is the unique zero of $\sigma_E$, that is, $x$ determines $y$. Assume now that $x, x'\in C$ are two distinct points such that $f(x)=f(x')\in E$, that is,

$$h^0\bigl(C, \eta_C(p+q-x-x')\bigr)=h^0\bigl(C, \eta_C(x+x'-p-q)\bigr)\geq 1.$$
We now argue that if $y,y'\in E$ are such that $H^0\bigl(X, \eta(x-y)\bigr)\neq 0$ and $H^0\bigl(X, \eta(x'-y')\bigr)\neq 0$, then $y=y'$, which will show that $C_1$ is the component of the limit Scorza curve lying in $C\times E$. We may write $H^0\bigl(C, \eta_C(x+x')\bigr)=\langle \sigma_C, \sigma_C'\rangle$, where $H^0\bigl(C, \eta_C(x)\bigr)\cong \mathbb C\cdot \sigma_C$ and $H^0\bigl(C, \eta_C(x')\bigr)\cong \mathbb C\cdot \sigma_C'$. By assumption
$$\bigl[\sigma_C(p), \sigma_C'(p)\bigr]=\bigl[\sigma_C(q), \sigma_C'(q)\bigr]\in E\cong \PP^1.$$
But this is precisely the condition that the \emph{same} section $\sigma_E\in H^0\bigl(E, \cO_E(1)\bigr)$ glues to both $\sigma_C$ and to $\sigma_C'$ at the points $p$ and $q$, thus showing that $y=y'$. A similar argument shows that the curve $C_2\subseteq E\times C$ is contained in the limit of the Scorza correspondence.

\vskip 4pt

Via a similar Mayer-Vietoris sequence, one shows that the case when at least one of the points $x, y$ specializes to $E\setminus \{p,q\}$ does not occur. When, on the other hand, $x$ specializes to $p$ and $y$ specializes to $q$, by blowing-up $\cX$ at the points $p$ and $q$, inserting exceptional divisors and writing down the Mayer-Vietoris sequence on the resulting central fibre, we easily concludes that this case is possible and corresponds to the intersection of $C_1$ and $C_2$ at the points $(p,q)$ and $(q,p)$.

Finally, we discuss the limit Scorza curve along the intersection of $C\times C$ and $C\times E$ (respectively, $E\times C$). In this case, we may assume, without loss of generality, that $y=p$ and $x\in C\setminus \{p,q\}$. We blow-up $\cX$ at $p$ and denote the exceptional divisor $E'$. Denoting the proper transforms of $C$ and $E$ by the same symbols, we consider the new central fibre
$$X'=C\cup E \cup E', \ \  \{p'\} := E\cap E', \ \ E\cap C=\{q\}, \ \ E'\cap C=\{p\},$$
where $y\in E'\setminus \{p,p'\}$ and the spin structure $\eta_{X'}$ is defined as the pull-back of $\eta$. From the exact sequence
$$0\lra H^0\bigl(\eta_{X'}(x-y)\bigr)\lra H^0\bigl(C, \eta_C(x)\bigr)\oplus H^0\bigl(E, \cO_E(1)\bigr)\oplus H^0\bigl(E', \cO_{E'}(-y)\bigr)\stackrel{\mathrm{ev}}\lra \mathbb C^3_{p,p',q},$$

we obtain that necessarily $H^0\bigl(C, \eta_C(x-p)\bigr)\neq 0$, that is, $x\in \{x_1, \ldots, x_{g-1}\}$, in which case the point $(x_i, p)\in S(C,\eta_C)$ can be identified with the point $(x_i,f(x_i))=(x_i,p)\in C_1$, where we have used (\ref{eq:identifB0}). This completes the proof.

\end{proof}

\subsection{The limit Scorza curve at a general point of $A_i$, when $i \geq 1$.}
The general point of the boundary divisor $A_i$ of $\ss_g^+$ corresponds to general $1$-pointed curves $[C, p]\in \cM_{i, 1}$ and $[D, q]\in \cM_{g-i, 1}$ together with \emph{even} theta-characteristics $\eta_C$ and $\eta_D$ on $C$ and $D$, respectively. To this data we associate the even spin curve of genus $g$

\begin{equation}\label{eq:genpointAi}
\begin{tikzcd}[column sep=22pt]
\bigl[X=C\cup E\cup D, \eta_{|C}\cong \eta_C,\  \ \eta_{|E}\cong \cO_E(1), \ \ \eta_{|D}\cong \eta_D\bigr]\in A_i,
\end{tikzcd}
\end{equation}
where $E$ is a smooth rational curve meeting $C$ at the point $p$ and $D$ at the point $q$.
We fix such a general point of $A_i$.

\vskip 3pt

We consider the Scorza curves $S(C, \eta_C)\subseteq C\times C$ and $S(D, \eta_D)\subseteq D\times D$. Using \cite[Theorem 4.1]{FV}, we may  assume that the curves $S(C,\eta_C)$ and $S(D,\eta_D)$ are smooth, and then of course $g\bigl(S(C,\eta_C)\bigr)=1+3i(i-1)$ and $g\bigl(S(D,\eta_D)\bigr)=1+3(g-i)(g-i-1)$.

Using Definition \ref{def:limitscorza}, it easily turns out that $S(C,\eta_C)$ (respectively $S(D,\eta_D)$) is the component of the limit Scorza curve lying in the component $C\times C$ (respectively $D\times D$)
of the cartesian product of the genus $g$ stable curve given by (\ref{eq:genpointAi}). We are left with determining the limit Scorza components in $C\times D$ and $D\times C$.

\vskip 4pt

Put $\{x_1, \ldots, x_i\}:=S(C, \eta_C)_p$ and, similarly, $\{y_1, \ldots, y_{g-i}\}:=S(D, \eta_D)_q$. For a general choice (\ref{eq:genpointAi}) of a point in $A_i$ the points $\{x_{\ell}\}_{\ell=1}^i$ and $\{y_k\}_{k=1}^{g-i}$ are pairwise distinct.

\vskip 4pt

Inside $C\times D$ we consider the ``vertical" copies $\{x_{\ell}\}\times D$, where $\ell=1, \ldots, i$, as well as the ``horizontal" copies $C\times \{y_k\}$ for $k=1, \ldots, g-i$. These intersect at the $i(g-i)$ points $(x_{\ell}, y_k)\in C\times D$. The point $(x_{\ell}, q)$ of the ``vertical copy" $\{x_{\ell}\}\times D\subseteq C\times D$ gets identified with the point $(x_{\ell}, p)$ of the curve $S(C,\eta_C)\subseteq C\times C$. The point $(p, y_k)$ of the  ``horizontal copy" $C\times \{y_k\}$ gets identified with the point $(q, y_k)$ of the curve $S(D, \eta_D) \subseteq D\times D$.

\vskip 4pt

By symmetry, inside $D\times C$ we consider the ``vertical" curves $\{y_k\}\times C$ for $k=1, \ldots, g-i$, as well as the ``horizontal" curves
$\{x_{\ell}\}\times D$ for $\ell=1, \ldots, i$, which intersect in the $i(g-i)$ points $(y_k, x_{\ell})\in D\times C$. The marked point $(y_k, p)$ of each component $\{y_k\}\times C$ is then identified with the marked point $(y_k,q)$ of the Scorza curve $S(D, \eta_D)$. Finally, the marked point $(x_{\ell}, q)$ of each component $\{x_{\ell}\}\times D$ is identified with the marked point $(x_{\ell}, p)$ of the Scorza curve $S(C,\eta_C)$.

\vskip 3pt

We are now in a position to describe the limiting Scorza curve for  a general point in $A_i$:

\begin{proposition}\label{prop:Ai}
The limit Scorza curve corresponding to a general point (\ref{eq:genpointAi}) of the boundary divisor $A_i$ is the following stable curve:

\begin{align*}
S(C,\eta_C) \ \ \bigcup \ \ \Bigl(\bigl(\{x_{1}\}\times D\bigr)\cup \ldots \cup \bigl(\{x_i\}\times D\bigr) \  \cup \ \bigl(C\times \{y_1\}\bigr)\cup \ldots \cup \bigl(C\times \{y_{g-i}\}\bigr)\Bigr)\  \bigcup\\
S(D, \eta_D) \ \ \bigcup \ \
\Bigl(\bigl(D\times \{x_1\}\bigr)\cup \ldots \cup \bigl(D\times \{x_i\}\bigr) \cup \bigl(\{y_1\}\times C\bigr)\cup \ldots \cup \bigl(\{y_{g-i}\}\times C\bigr)\Bigr).
\end{align*}

\begin{figure}[!h]
\begin{center}
\hspace{2cm}
\includegraphics[width=8cm, height=7cm]{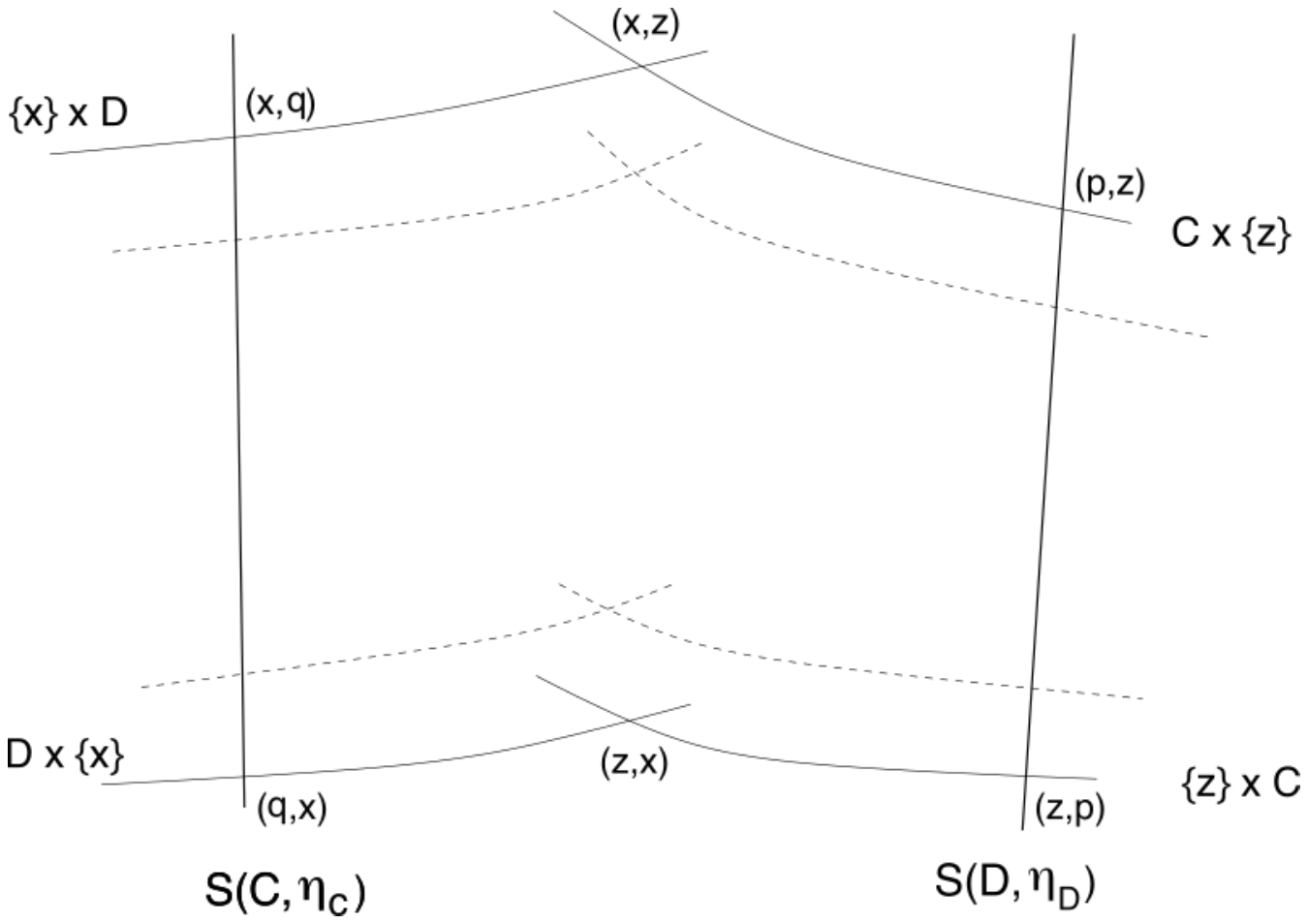}
\caption{The Scorza curve corresponding to a general point of $A_i$}
\end{center}
\end{figure}

This decomposition corresponds to the components $C\times C$, $C\times D$, $D\times D$ and $D\times C$ (in this order) of the cartesian product.
\end{proposition}

\begin{remark}\label{rem:Ai}
Note that none of the components of the Scorza curve vary in moduli as the point of attachment $p$ (respectively $q$) varies on either $C$ (respectively $D$). Observe also that the stable curve described in Proposition \ref{prop:Ai} has $\nu:=2+2(i+g-i)=2+2g$ irreducible components, all smooth. The sum  of their genera equals
\begin{align*}
\sigma:=g\bigl(S(C,\eta_C)\bigr)+g\bigl(S(D,\eta_D)\bigr)+2\bigl((g-i)i+i(g-i)\bigr)\\
=2+3i(i-1)+3(g-i)(g-i-1)+4i(g-i)=3g^2-2gi+2i^2-3g+2.
\end{align*}
These components meet at a total number of $\delta:=2i(g-i)+2(g-i)+2i=2(ig-i^2+g)$ nodes. It follows that the arithmetic genus of the curve described in Proposition \ref{prop:Ai} equals $p_a\bigl(S(X,\eta))=1+\sigma+\delta-\nu=1+3g(g-1)$, which provides a numerical verification of the conclusion of Proposition \ref{prop:Ai}.
\end{remark}

\noindent \emph{Proof of Proposition \ref{prop:Ai}}. To simplify notation, we will identify the marked point $p\in C$ and $q\in D$, respectively. We consider the curve of compact type $C\cup_p D$ as above and assume a pair $(x,y)\in (C\cup D)\times (C\cup D)$ lies in the limiting Scorza correspondence $S(X, \eta)$. We may assume that $x\neq y$ and that $p\notin \{x,y\}$. If  $x,y\in C$ (respectively, if $x,y\in D$), we quickly reach the conclusion that $(x,y)\in S(C, \eta_C)$ (respectively, that $(x,y)\in S(D, \eta_D)$). We may assume thus without loss of generality that $x\in C\setminus\{p\}$ and that $y\in D\setminus \{q\}$. The limiting condition on $(x,y)$ can be stated as saying that there exist a limit linear series of type $\mathfrak g^0_{g-1}$ on $C\cup D$ (in the sense of \cite{EH}), obtained by twisting the spin structure by the Cartier divisor $x-y$ on the stable curve $C\cup D$. This means that there exists sections $0\neq \sigma_C\in H^0\bigl(C,\eta_C((g-i-1)\cdot p+x)\bigr)$ and $\sigma_D\in H^0\bigl(D, \eta_D((i+1)\cdot q-y)\bigr)$ such that
\begin{equation}\label{eq:ordersAi}
\mbox{ord}_p(\sigma_C)+\mbox{ord}_q(\sigma_D)\geq g-1.
\end{equation}

\vskip 3pt
We claim that $(x,p)\in S(C,\eta_C)$, that is, $x\in \{x_1, \ldots, x_i\}$, or else $(y,q)\in S(D, \eta_D)$, that is, $y\in \{y_1, \ldots, y_{g-i}\}$. Indeed, if neither of these possibilities is realized, then necessarily $\mbox{ord}_p(\sigma_C)\leq g-i-1$ \emph{and} $\mbox{ord}_q(\sigma_D)\leq i-1$, thus $\mbox{ord}_p(\sigma_C)+\mbox{ord}_q(\sigma_D)\leq g-2$, which contradicts (\ref{eq:ordersAi}). Assuming for instance $x=x_{\ell}\in S(C,\eta_C)_p$ for $\ell\in \{1, \ldots, i\}$, then $\mbox{ord}_p(\sigma_C)=g-i$, in which case $\mbox{ord}_q(\sigma_D)\geq i-1$, which is automatic by Riemann-Roch,  showing that $y\in D$ can be chosen arbitrarily. This accounts for the  component $\{x_{\ell}\}\times D\subseteq C\times D$ inside the limiting Scorza curve $S(X,\eta)$. The remaining components can be explained similarly.
\hfill $\Box$

\section{The Scorza curve at a general point of the divisor $B_i$, when $i\geq 1$}\label{sect:scorza_Bi}

Determining the limit Scorza curve along the boundary divisor $B_i$ of $\ss_g^+$ turns out to be more delicate than in the case of the divisor $A_i$. We need some preparation and start with a smooth pointed curve $[C, p]\in \cM_{g,1}$ and with an \emph{odd} theta characteristic on $C$.

We have the exact sequence in cohomology on $C^{(2)}$
$$0\longrightarrow H^0\bigl(\cL_{\eta(p)}(-\delta)\bigr) \lra H^0\bigl(\cL_{\eta(p)}(\delta)\bigr)\stackrel{\rho}{\lra} H^0\bigl(C, \cO_C(2p)\bigr)\lra H^1\bigl(\cL_{\eta(p)}(-\delta)\bigr),$$
where we have used that $\cO_{\overline{\Delta}}\bigl(\cL_{\eta(p)}(\delta)\bigr)\cong \eta^{\otimes 2}\otimes \omega_C^{\vee}(2p)\cong \cO_C(2p)$.

\begin{lemma}\label{lemma:etap}
Let $[C, p]\in \cM_{g,1}$ be a smooth pointed curve and fix an \emph{odd} spin structure $\eta$ on $C$ with $h^0(C,\eta)=1$ and assume $p\not\in\mathrm{supp}(\eta)$. The natural map
$$H^0\bigl(C^{(2)}, \cL_{\eta(p)}(\delta)\bigr) \stackrel{\rho}{\lra} H^0(C, \cO_C(2p)\bigr)$$
above is an isomorphism. In particular, if $p\in C$ is not a hyperrelliptic Weierstrass point of $C$, then the linear system $\bigl|\cL_{\eta(p)}(\delta)\bigr|$ contains a unique curve $Y_{\eta, p} \subseteq C^{(2)}$.

When $p\in \mathrm{supp}(\eta)$, the map $\rho$ is zero. However, $ \bigl|\cL_{\eta(p)}(\delta)\bigr|$ still contains exactly one curve $Y_{\eta, p} \subseteq C^{(2)}$.
\end{lemma}
\begin{proof}
We use the identification provided by (\ref{eq:cohsympr}) and write
\begin{align*}
H^0\bigl(C^{(2)}, \cL_{\eta(p)}(-\delta)\bigr)\cong \bigwedge^2 H^0\bigl(C, \eta(p)\bigr) \ \  \mbox{ and } \\
 H^1\bigl(C^{(2)}, \cL_{\eta(p)}(-\delta)\bigr)\cong H^1\bigl(C, \eta(p)\bigr)\otimes H^0\bigl(C, \eta(p)\bigr).
\end{align*}
If $p\notin \mbox{supp}(\eta)$, then $h^0(C, \eta(p))=1$ and $H^1(C,\eta(p))=0$, therefore $H^i\bigl(C^{(2)}, \cL_{\eta(p)}(-\delta)\bigr)=0$, for $i=0,1$,
which establishes the canonical identification
\begin{equation}\label{eq:biid}
H^0\bigl(C^{(2)}, \cL_{\eta(p)}(\delta)\bigr)\cong H^0\bigl(C, \cO_C(2p)\bigr).
\end{equation}
If, on the other hand, $p\in \mbox{supp}(\eta)$, then $H^1\bigl(C^{(2)}, \cL_{\eta(p)}(-\delta))\cong \mbox{Hom}\bigl(H^0(\eta(-p)), H^0(\eta(p))\bigr)$, hence we conclude that the map $H^0\bigl(C, \cO_C(2p)\bigr)\rightarrow H^1\bigl(C^{(2)}, \cL_{\eta(p)}(-\delta)\bigr)$ is injective and $\rho$ is $0$. Therefore $H^0\bigl(C^{(2)},\cL_{\eta(p)}(\delta)\bigr)\cong \bigwedge^2 H^0\bigl(C, \eta(p)\bigr)$ (also see Appendix A.3 of \cite{I}, noting that one can use Serre Duality and the fact that $\omega_{C^{(2)} }\otimes (\cL_{\eta(p)}(\delta))^{-1} \cong
\cL_{\eta(-p)}(-2\delta)$).
In both cases it follows that the linear system $\bigl|\cL_{\eta(p)}(\delta)\bigr|$ contains a single curve.
\end{proof}


\vskip 4pt

We denote by $X_{\eta,p}:=q^{-1}(Y_{\eta,p})$ the inverse image in $C\times C$ of the curve $Y_{\eta, p}\subseteq C^{(2)}$ singled out by Lemma \ref{lemma:etap}. Observe  that $\cO_{C\times C}(X_{\eta,p})\cong \eta(p)\boxtimes \eta(p)(\Delta)$. Therefore, by the adjunction formula, we obtain that $p_a(X_{\eta,p})=3g^2+g$. Similarly, we obtain that
\begin{equation}\label{eq:genyp}
p_a(Y_{\eta,p})=\frac{g}{2}(3g+1).
\end{equation}
We will show that in general  $Y_{\eta,p}$ is a smooth curve. We begin with the following result:

\begin{proposition}\label{lem:sing2}
Fix a general point $[C, \eta, p]\in \cS_{g,1}^-$. Suppose $x+y\in Y_{\eta , p}$ with $x\neq y$ and none of the points $x,y$ belong to $\mathrm{supp}(\eta)$. Then $Y_{\eta, p}$ is nonsingular at the point $x+y$.
\end{proposition}

\begin{proof}
It is sufficient to prove that $H^0\bigl(C^{(2)}, \cI_{\{x+y\}/ C^{(2)}}^2\otimes \cL_{\eta(p)}(\delta)\bigr)=0$. We have the following Koszul resolution of the ideal $\cI_{\{x+y\}/ C^{(2)}}^2$

\begin{align*} 0 \lra \cO_{C^{(2)}} (- C_{2x+y}) \bigoplus \cO_{C^{(2)}} (- C_{x+2y})\lra \\
 \cO_{C^{(2)}} (- C_{2x}) \bigoplus \cO_{C^{(2)}} (- C_{x+y}) \bigoplus \cO_{C^{(2)}} (- C_{2y})
\lra \cI_{\{x+y\}/ C^{(2)}}^2 \lra 0,
\end{align*}
which, after twisting by the line bundle $\cL_{\eta(p)}(\delta)$, becomes

\begin{equation}\label{eq:nI2xyres}
\begin{split}
0 \lra \cL_{\eta(p-2x-y)}(\delta) \bigoplus \cL_{\eta(p-x-2y)}(\delta) \lra \cL_{\eta(p -2x)}(\delta) \bigoplus \cL_{\eta(p-x-y)}(\delta) \bigoplus
\cL_{\eta(p -2y)}(\delta) \\
\lra \cI_{\{x+y\}/C^{(2)}}^2\otimes \cL_{\eta(p)}(\delta)\lra 0.
\end{split}
\end{equation}

Since $p,x$ and $y$ do not belong to $\mathrm{supp}(\eta)$, by using suitable twists of the exact sequence
$0\rightarrow \cO_{C^{(2)}}(-2\delta)\rightarrow \cO_{C^{(2)}}\rightarrow \cO_{\overline{\Delta}}\rightarrow 0$, via the identifications  (\ref{eq:cohsympr}), we obtain (also see Appendix A.3 of \cite{I})
$$H^0\bigl(C^{(2)}, \cL_{\eta(p-2x)}(\delta)\bigr)=0, \ \ H^0\bigl(C^{(2)}, \cL_{\eta(p-2y)}(\delta)\bigr)=0 \mbox{ and } H^0\bigl(C^{(2)}, \cL_{\eta(p-x-y)}(\delta)\bigr)=0.$$

We therefore obtain from (\ref{eq:nI2xyres}) the following exact sequence
\[
\begin{split} 0 \lra H^0\bigl(\cI_{\{x+y\}/C^{(2)}}^2\otimes \cL_{\eta(p)}(\delta)\bigr) \lra
H^1(\cL_{\eta(p-2x-y)}(\delta)) \oplus H^1(\cL_{\eta(p-x-2y)}(\delta)) \\
\lra H^1(\cL_{\eta(p -2x)}(\delta)) \oplus H^1(\cL_{\eta(p-x-y)}(\delta)) \oplus H^1(\cL_{\eta(p -2y)}(\delta))\lra \cdots
\end{split}
\]
which, after applying Serre duality, yields by taking duals the following exact sequence:
\[
\begin{split}
H^1\bigl(\cL_{\eta(-p +2x)}(-2\delta)\bigr) \oplus H^1\bigl(\cL_{\eta(-p+x+y)}(-2\delta)\bigr) \oplus H^1\bigl(\cL_{\eta(-p +2y)}(-2\delta)\bigr)\stackrel{\gamma} \lra \\
H^1\bigl(\cL_{\eta(-p+2x+y)}(-2\delta)\bigr) \oplus H^1\bigl(\cL_{\eta(-p+x+2y)}(-2\delta)\bigr)
 \lra H^0\bigl(\cI_{\{x+y\}/C^{(2)}}\otimes \cL_{\eta(p)}(\delta)\bigr)^{\vee} \lra 0 .\end{split}
\]
We now consider the following morphism of exact sequences
\begin{equation}
\label{eq:alt-sep}
\begin{tikzcd}[column sep=18pt, row sep=4pt]
  \mbox{Sym}^2 H^0(\eta(-p+2x))  & H^0(\omega_C(-2p+4x))  & H^1(\cL_{\eta(-p+2x)}(-2\delta)) \\
  \oplus & \oplus & \oplus \\
  \mbox{Sym}^2 H^0(\eta(-p+2y)) \arrow[r, "\mu"]
  & H^0(C, \omega_C(-2p+4y))  \arrow[r]
  & H^1(\cL_{\eta(-p+2y)}(-2\delta))\\
  \oplus & \oplus &  \oplus
    \\
   \mbox{Sym}^2 H^0(\eta(-p+x+y)) \arrow[ddddd, "\alpha"] & H^0(\omega_C(-2p+2x+2y)) \arrow[ddddd, "\beta"] & H^1(\cL_{\eta(-p+x+y)}(-2\delta)) \arrow[ddddd, "\gamma"]\\
              &  &                     \\
              &   &                    \\
              &   &                    \\
              &   &                    \\
  \mbox{Sym}^2 H^0(\eta(-p+2x+y)) \arrow[r, "j"] & H^0(\omega_C(-2p+4x+2y))\arrow[r, two heads] & H^1(\cL_{\eta(-p+2x+y)}(-2\delta))\\
  \oplus  & \oplus  & \oplus\\
  \mbox{Sym}^2 H^0(\eta(-p+x+2y)) & H^0(\omega_C(-2p+2x+4y)) & H^1(\cL_{\eta(-p+x+2y)}(-2\delta))\\
\end{tikzcd}
\end{equation}
where $\mu$ is the direct sum of the corresponding multiplication maps at the level of global sections, whereas $\beta$ is given componentwise by $\beta(a, b,c)=(a-c, b-c)$. The  last horizontal arrow  in diagram (\ref{eq:alt-sep}) is surjective because
$$H^1\bigl(C^{(2)}, \cL_{\eta(-p+2x+y)}\bigr)=0 \mbox{  and }  H^1\bigl(C^{(2)}, \cL_{\eta(-p+x+2y)}\bigr)=0,$$
 which follows because of our generality assumptions via the identifications (\ref{eq:cohsympr}). To conclude that the map $\gamma$ is surjective it remains to show via (\ref{eq:alt-sep}) that
\begin{equation}\label{eq:sufficient}
\mathrm{Im}(j)+\mathrm{Im}(\beta)=H^0(C, \omega_C(-2p+4x+2y))\bigoplus H^0(C, \omega_C(-2p+2x+4y)).
\end{equation}

Since the summands $\mbox{Sym}^2 H^0 (C, \eta(-p + 2x + y))$ and $H^0 (C, \omega_C(-2p +4x))$ map via the maps $j$, respectively, $\beta$ only into the summand $H^0 \bigl(C, \omega_C( -2p + 4x + 2y)\bigr)$ (a similar conclusion holds if we interchange $x$ and $y$),  a sufficient condition for (\ref{eq:sufficient}) to hold is that

\begin{equation}\label{eq:sufficient2}
 j\Bigl(\mbox{Sym}^2  H^0 (C, \eta( -p + 2x + y))\Bigr)+ H^0\bigl (C, \omega_C(- 2p +4x)\bigr)=H^0 \bigl(C, \omega_C(-2p + 4x + 2y)\bigr)
 \end{equation}

 and
 \begin{equation}\label{eq:sufficient3}
 j\Bigl(\mbox{Sym}^2  H^0 (C, \eta( -p + x +2 y))\Bigr)+ H^0\bigl (C, \omega_C(- 2p +4y)\bigr)=H^0 \bigl(C, \omega_C(-2p + 4y + 2x)\bigr).
 \end{equation}
 Note that $h^0(C,\omega_C(-2p +4x))= g+1$ and
 $h^0 (C, \omega_C(-2p+4x+2y))= g+3$.

\vskip 4pt

As observed earlier, the assumption that $p,x,y$ do not belong to the support of $\eta$ implies $h^0(C, \eta(-p + 2x + y))=2$, therefore
$\mbox{dim } \mbox{Sym}^2 H^0(C, \eta(-p+2x+y))=3$ and (\ref{eq:sufficient2}) amounts to proving that the intersection

\begin{equation}\label{eq:intersection2}
j\Bigl(\mbox{Sym}^2 H^0(C, \eta(-p+2x+y))\Bigr) \bigcap H^0\bigl(C, \omega_C(-2p+4x)\bigr)\subseteq H^0\bigl(C, \omega_C(-2p+4x+2y)\bigr)
\end{equation}
is $1$-dimensional. Similarly, the equality (\ref{eq:sufficient3}) amounts to the statement obtained from (\ref{eq:intersection2}) if we interchange $x$ and $y$.

Observe that  $H^0\bigl(C,\omega_C(-2p+4x)\bigr)$ is the subspace of $H^0\bigl(C,\omega_C(-2p+4x +2y)\bigr)$
consisting of sections vanishing with multiplicity at least $2$ at $y$. We choose a basis $(\sigma_1, \sigma_2)$ of
$H^0(C, \eta(-p+2x+y))$, where $\mbox{ord}_y(\sigma_2)\geq 1$ and $\mbox{ord}_y(\sigma_1)=0$.  Assume there is a further element other than $j(\sigma_2^2)$ in the intersection
(\ref{eq:intersection2}). Without loss of generality we may assume it to be $j(\sigma_1^2+\sigma_1\cdot \sigma_2+\sigma_2^2)$.  We obtain $\mbox{ord}_y(\sigma_1^2+\sigma_1\cdot \sigma_2+\sigma_2^2)\geq 2$, which
implies $\mbox{ord}_y(\sigma_1)\geq 1$, that is, $h^0(C, \omega_C(-p+2x))=2$, which is a contradiction. This finishes the proof.
\end{proof}

\begin{theorem}\label{cor:genusBi}
For a general choice of $[C,\eta,p]\in \cS_{g,1}^-$, the curve  $X_{\eta,p}$ has a unique nodal singularity at the point $(p,p)$, whereas $Y_{\eta,p}$ is a smooth curve.
\end{theorem}
\begin{proof}
First, we  show that $X_{\eta, p}$ is singular at the point $(p,p)\in C\times C$ and, to that end, we prove that the intersection multiplicity at
$(p,p)$ of  $X_{\eta,p}$ with both the diagonal $\Delta$ and with the translated curve $C\times \{p\}$ is at least two.

\vskip 3pt

We may assume that $p\in C$ is not a hyperelliptic Weierstrass point and $p\notin \mathrm{supp}(\eta)$. We have the following exact sequence
$$0\longrightarrow H^0\bigl(C^{(2)}, \cI_{\{2\cdot p\}/C^{(2)}}\otimes \cL_{\eta(p)}(\delta)\bigr)\lra H^0\bigl(C^{(2)}, \cL_{\eta(p)}(\delta)\bigr)\lra
H^0\bigl(\cO_p(2p)\bigr).$$
Using (\ref{eq:biid}) and the fact that since $H^0(C, \cO_C(p))\cong H^0(C, \cO_C(2p))$, the map $$H^0\bigl(C, \cO_C(2p)\bigr)\rightarrow H^0\bigl(\cO_{p}(2p)\bigr)$$ is equal to zero, therefore $Y_{\eta,p}$ passes through the point $\tilde{p}:=2\cdot p\in \overline{\Delta}$.

\vskip 4pt

We consider the embedding $\iota\colon C\hookrightarrow  C^{(2)}$ given by $\iota(x):=p+x$ so that $\mbox{Im}(\iota)=C_p$ (see Section \ref{subsecC_D}). Then $\iota^*\bigl(\cL_{\eta(p)}(\delta)\bigr)=\eta(2p)\in W^1_{g+1}(C)$, that is, the intersection cycle $C_p\cdot Y_{\eta,p}$ can be regarded as an element of the pencil $|\eta(2p)|$ that passes through the point $p$. Since $h^0(C, \eta(p))=1$,  one must have
$$C_p\cdot Y_{\eta,p}=2p+\mathrm{supp}(\eta)\in C^{(g+1)},$$
that is, $\bigl(Y_{\eta,p}\cdot C_p)_{\tilde{p}}= 2$ (since $p\not\in \mathrm{supp}(\eta)$). It immediately follows that we also have $\bigl(X_{\eta,p}\cdot C\times \{p\}\bigr)_{{(p,p)}}= \bigl(X_{\eta,p}\cdot  \{p\}\times C\bigr)_{{(p,p)}}=2$

\vskip 3pt

Next we determine  the intersection of $X_{\eta,p}$ with  $\Delta$. Denoting by $j\colon C\rightarrow C\times C$ the embedding $j(x):=(x,x)$,  since $j^*(\cL_{\eta(p)})\cong \cO_C(2p)$ and $h^0(C, \cO_C(2p))=1$, it follows that
$$\bigl(\Delta\cdot X_{\eta,p}\bigr)_{(p,p)}=2.$$
Since the curves $\Delta$ and $C\times \{p\}$ have distinct tangents at $(p,p)$, it follows that $X_{\eta,p}$ is singular at $(p,p)$. On the other
hand the curves $\overline{\Delta}$ and $C_p$ are tangent at the point $2\cdot p\in C^{(2)}$, so the above argument only shows that $Y_{\eta,p}$ is tangent to $\overline{\Delta}$ at $\tilde{p}$.

\vskip 4pt

In order to show that $Y_{\eta,p}$ has no other singularity, we specialize to the situation when $p\in \mathrm{supp}(\eta)$, in which case $\eta(p)\in W^1_g(C)$ is a pencil. We consider the trace curve
\begin{equation}\label{eq:tr2}
\Gamma_{\eta(p)}:=\Bigl \{ x+y\in C^{(2)}: H^0\bigl(C, \eta(p-x-y)\bigr)\neq 0\Bigr\}.
\end{equation}
Then, via \cite[Lemma 2.1]{I}, we observe that $\cO_{C^{(2)}}(\Gamma_{\eta(p)}) \cong \cL_{\eta(p)}(-\delta)$. Since, via Lemma \ref{lemma:etap}, we have that $h^0\bigl(C^{(2)}, \cL_{\eta(p)}(\delta)\bigr)=1$, it follows that the curve $Y_{\eta,p}$ splits into the union of $\Gamma_{\eta(p)}$ and $\overline{\Delta}$. Assuming that $\eta(p)$ is simply ramified, the union $Y_{\eta,p} \cup\overline{\Delta}$ is transverse. Furthermore, using again \cite{vdGK2}, the curve $\Gamma_{\eta(p)}$ is smooth, therefore the union $\Gamma_{\eta(p)} \cup \overline{\Delta}$ is a stable curve. In particular, $Y_{\eta, p}$ is a nodal curve for a general choice of $[C, \eta,p]$.

\vskip 4pt

Assume $x_0+y_0 \in Y_{\eta,p}$ is a singularity, where $x_0\in \mathrm{supp}(\eta)$ and $y_0\in C\setminus \{x_0, p\}$ (see Proposition \ref{lem:sing2}). Considering the embedding
$j_{x_0}\colon C\hookrightarrow C^{(2)}$ given by $j_{x_0}(y):=x_0+y$, since $j_{x_0}^*\bigl(\cO_{C^{(2)}}(Y_{\eta,p})\bigr)=\eta(p+x_0)$,
and  $|\eta(p+x_0)|=|\eta(x_0)|+p$, it follows that $y_0$ is one of the ramification points of the pencil $\eta(x_0)\in W^1_g(C)$. Assume first $y_0\notin \mathrm{supp}(\eta)$. Since $\bigl(C_{y_0}\cdot Y_{\eta,p}\bigr)_{x_0+y_0}\geq 2$, we obtain that $H^0(C, \eta(p+y_0-2x_0))\neq 0$. It follows that $H^0(C, \eta(x_0-2y_0-p))\neq 0$, that is, $p$ has to be one of the anti-ramification points\footnote{For a cover $f\colon C\rightarrow \PP^1$, we say that a point $p\in C$ is an anti-ramification point if it lies in a fibre of $f$ over a branch point of $f$.} of the pencil $\eta(x_0)$. Since there are finitely many such pencils, this singles out a finite number of points $p$ and, for a general choice of $p\in C$, this scenario does not occur. If, on the other hand $y_0\in \mathrm{supp}(\eta)$, we obtain that
\begin{equation}\label{eq:double_supp}
H^0(C, \eta(x_0-2y_0)\neq 0\ \ \mbox{ and } \ \ H^0(C, \eta(y_0-2x_0)\neq 0.
\end{equation}
In this case, $\eta(x_0)$ has $y_0$ as a base point, which is impossible for a general $[C, \eta]$, which implies that (\ref{eq:double_supp}) is impossible. It follows that $Y_{\eta,p}$ has no singularities at a point $(x_0, y_0)$, where $x_0\in \mathrm{supp}(\eta) \setminus \{p\}$. Combining this with Proposition \ref{lem:sing2}, we conclude that the last possible case to be ruled out is when $Y_{\eta, p}$ has a nodal singularity at the point $\tilde{p}$.
We consider the cover $q\colon X_{\eta,p}\rightarrow Y_{\eta,p}$ and recall that $(p,p)$ is a nodal singularity of $X_{\eta,p}$. Since $q$ is \'etale everywhere on $X_{\eta,p}\setminus \{(p,p)\}$, we reach a contradiction with the fact $p_a(X_{\eta,p})=2p_a(Y_{\eta,p})$, which finishes the proof.
\end{proof}

\begin{corollary}\label{corr:fibre_minus}
Let $g\geq 3$. For a general choice of $[C, \eta]\in \cS_g^-$, the curve $Y_{\eta,p}$ is nodal for every $p\in C$.
\end{corollary}
\begin{proof} It suffices to combine Proposition \ref{lem:sing2} and Theorem \ref{cor:genusBi}. It follows that $Y_{\eta, p}$ is smooth when $p\notin \mathrm{supp}(\eta)$, whereas, when $p\in \mathrm{supp}(\eta)$, the curve $Y_{\eta,p}$ is the transverse union of the smooth trace curve $\Gamma_{\eta(p)}$ defined by (\ref{eq:tr2}) and the diagonal $\overline{\Delta}$ of $C^{(2)}$.
\end{proof}

\begin{remark}\label{rem:Bi}
The double cover $q\colon X_{\eta,p}\rightarrow Y_{\eta,p}$ is not admissible in the sense of \cite{HM}. We describe the associated admissible double cover. We take the normalization $\nu\colon X_{\eta,p}'\rightarrow X_{\eta,p}$ and let $\{p^+, p^-\}:=\nu^{-1}\bigl\{(p,p)\bigr\}$. We attach a smooth rational curve $E$ to the smooth curve $Y_{\eta,p}$ at the point $\tilde{p}=2\cdot p\in Y_{\eta, p}$. Then the associated double cover is given by
\begin{equation}\label{eq:admdob6}
\Bigl[f\colon X_{\eta,p}' \cup E'\lra Y_{\eta,p}\cup_{\tilde{p}}  E,\ \  f^{-1}(\tilde{p})=\{p^+, p^-\}\Bigr],
\end{equation}
where $f_{| X_{\eta,p}'}=q\circ \nu$ and $E'$ is a smooth rational curve meeting $X_{\eta,p}'$ at the points $p^+$ and $p^-$ with $f_{|E'}\colon E'\rightarrow E$ being a double cover mapping the points $p^+$ and $p^-$ to $\tilde{p}\in E$.
\end{remark}

\vskip 4pt

With this preparation in place, we now determine the limit Scorza correspondence of a general point of $B_i$ corresponding to pointed curves $[C, p]\in \cM_{i,1}$ and $[D,q]\in \cM_{g-i,1}$ together with odd theta-characteristics $\eta_C$ on $C$ and $\eta_D$ on $D$.
We write $\mathrm{supp}(\eta_C)=\{x_1, \ldots, x_{i-1}\}$ and $\mathrm{supp}(\eta_D)=\{y_1, \ldots, y_{g-i}\}$.

\vskip 4pt

The spin structure $\eta=\bigl(\eta_C, \cO_{E}(1), \eta_D\bigr)$ on the  curve $X:=C\cup_{p\sim 0} E \cup_{q\sim \infty}D$ has a two-dimensional space of global sections. Imposing the simple-minded limit of the condition $h^0 (X, \eta (x-y)) > 0$ on the pair $[X, \eta]$, we obtain that the limit of the Scorza correspondence contains several two-dimensional components. We therefore need to once more take into account that the Scorza correspondence is the inverse image of a curve in the symmetric product of $C\cup D$.

\vskip 3pt

Consider a one-parameter family of even spin curves $(\cX, \eta_{\cX}) \ra (B,0)$ with central fiber $(X, \eta)$ as above and generic fiber a general element of $\cS_g^+$. We choose a point $(x,y)$ of the central fibre lying in the limiting Scorza curve. Suppose first $x,y\in C\setminus \{p\}$. As in the proof of Proposition \ref{prop:Ai}, there exist sections $0\neq \sigma_C\in H^0\bigl(C, \eta_C((g-i)\cdot p+x-y)\bigr)$ and $0\neq \sigma_D\in H^0\bigl(D, \eta_D(i\cdot q)\bigl)$ such that $\mbox{ord}_p(\sigma_C)+\mbox{ord}_q(\sigma_D)\geq g-1$. Since $p\notin \mathrm{supp}(\eta_D)$, we obtain that
$\mbox{ord}_q(\sigma_D)\leq i$, therefore $\mbox{ord}_p(\sigma_C)\geq g-i-1$, which leads to the condition
$$H^0\bigl(C, \eta_C(p+x-y)\bigr)\neq 0,$$
which is automatically satisfied for arbitrary points $x,y\in C$. However, the component of the limiting Scorza correspondence lying in $C\times C$ must also be the pull-back of a curve in $C^{(2)}$, and this curve lies necessarily in $\bigl|\cL_{\eta_C(p)}(\delta_C)\bigr|$, therefore, using Lemma \ref{lemma:etap}, it must be the curve $X_{\eta_C,p}$. A similar argument yields that the component of the limiting Scorza correspondence lying in $D\times D$ is the curve $X_{\eta_D,q}$.

\vskip 4pt

Assume that $x\in C\setminus\{p\}$ and $y\in D\setminus \{q\}$. Then there exist non-zero sections
\begin{align*}
\sigma_C\in H^0\bigl(C, \eta_C((g-i-1)\cdot p+x)\bigr), \ \ \sigma_D\in H^0\bigl(D, \eta_D((i+1)\cdot q-y)\bigr),\\
\mbox{and } \ \sigma_E\in H^0\bigl(E, \cO_{E}((g-i-1)\cdot p+(i+1)\cdot q)\bigr),
\end{align*} with $\mbox{ord}_p(\sigma_C)+\mbox{ord}_p(\sigma_E)\geq g-1$ and $\mbox{ord}_q(\sigma_E)+\mbox{ord}_q(\sigma_D)\geq g-1$. In this situation either $y\in \mathrm{supp}(\eta_D)$, leading to $\mbox{ord}_q(\sigma_D)=i$ and $\mbox{ord}_p(\sigma_C)=g-i-1$ (and $\sigma_E$ being uniquely determined up to a constant by the conditions $\mbox{ord}_p(\sigma_E)=i$ and $\mbox{ord}_q(\sigma_E)=g-i-1$), or else $x\in \mathrm{supp}(\eta_C)$. This accounts for the components $C\times \{y_k\}$, respectively, $\{x_{\ell}\}\times D$ of the limiting Scorza curve.

\vskip 4pt

A similar calculation shows that the cases $x\in C\setminus \{p\}$, $y\in E\setminus\{p,q\}$, or, respectively, $x\in D\setminus \{q\}$, $y\in E\setminus \{p,q\}$ do not appear. Finally, we deal with the case when both $x\neq y$ lie on $E$. Then there exist non-zero sections
\begin{align*}
\sigma_C\in H^0\bigl(C, \eta_C((g-i)\cdot p\bigr), \ \ \sigma_D\in H^0\bigl(D, \eta_D(i\cdot q)\bigr)\\
\mbox{ and } \sigma_E\in H^0\bigl(E, \cO_E(1)((g-i-1)\cdot p+(i-1)\cdot q+x-y)\bigr).
\end{align*}
We get $\mbox{ord}_p(\sigma_C)\leq g-i$ and $\mbox{ord}_q(\sigma_D)\leq i$, therefore $\mbox{div}(\sigma_E)\geq (i-1)\cdot p+(g-i-1)\cdot q$.
Given $x\in E$, the point $y$ is uniquely determined by the condition that $\sigma_E$ vanishes at $y$.

\vskip 3pt

The curves $X_{\eta_C,p}$ and $X_{\eta_D,q}$ are nodal at the points $(p,p)$ and $(q,q)$, which get identified. We consider their normalizations $X_{\eta_C,p}'$ and $Y_{\eta_D,q}'$ with the points $p^+, p^-\in X_{\eta_C,p}$, respectively, $q^+,q^-\in X_{\eta_D,q}$, as described in Remark \ref{rem:Bi}.  We join $X_{\eta_C,p}'$
and $X_{\eta_D,q}'$ by identifying $p^+$ and $q^+$, respectively, $p^-$ and $q^-$.

\vskip 3pt

\begin{proposition}\label{prop:limitBi}
The limit Scorza curve at a generic point $[C\cup D, \eta_C, \eta_D]$ of $B_i$ is the transverse union of the following curves in the Cartesian power $C^2 \cup C\times D \cup D\times C \cup D^2$.

\begin{align*}
X_{\eta_C,p}' \ \ \bigcup \ \ \Bigl((\eta_C \times D) \cup (C \times \eta_D)\Bigr) \
\bigcup \ \Bigl((D\times \eta_C)\cup (\eta_D \times C)\Bigr) \ \ \bigcup \ \ X_{\eta_D,q}'.
\end{align*}
The decomposition corresponds to the components of the Scorza curve in $C\times C$, $C\times D$, $D\times C$ and $D\times D$ (in this order).
\end{proposition}
\begin{remark}\label{rem:Bi}
The  curve exhibited in Theorem \ref{prop:limitBi} has $\nu=2(i-1+g-i-1)+2=2g-2$ irreducible components and $\delta=2(i-1)(g-i-1)+2(i-1+g-i-1)+2=2i(g-i)$ nodes. The sum of the genera of its irreducible components equals
\begin{align*}
\sigma=p_a\bigl(X'_{\eta_C, p}\bigr)+p_a\bigl(X'_{\eta_D, q}\bigr)+2(i-1)(g-i)+2(g-i-1)i.
\end{align*}
Taking into account that $p_a\bigl(X'_{\eta_C, p}\bigr)=3i^2+i-1$ and $p_a\bigl(X'_{\eta_D, q}\bigr)=3(g-i)^2+g-i-1$, we verify that the arithmetic genus of the limiting Scorza curve equals $1+\sigma-\nu+\delta=1+3g(g-1)$.
\end{remark}

\section{The weight of the Szeg\H{o}-Hodge class on $\ss_g^+$}

We shall determine the Szeg\H{o}-Hodge class on $\ss_g^+$, that is, the pull-back of the Hodge class under the Scorza map $\mathfrak{Sc}\colon \ss_g^+\dashrightarrow \mm_{1+3g(g-1)}$. We recall that via Diagram (\ref{eq:diagram_stacks}), the morphism $\mathfrak{Sc}$ factors through the morphism
\begin{equation}\label{eq:def_xi}
\xi\colon \ss_g^+\dashrightarrow \rr_{1+\frac{3}{2}g(g-1)}, \ \ \  [C, \eta]\mapsto \bigl[S(C,\eta)\rightarrow T(C,\eta)\bigr],
\end{equation} that is, $\chi \circ \xi=\mathfrak{Sc}$.
We denote by $\lambda'$, $\delta_i'$ the standard generators of $\mbox{Pic}(\mm_{1+3g(g-1)})$,  where $i=0, \ldots, \left \lfloor \frac{1+3g(g-1)}{2}\right \rfloor$.  Furthermore, let $\lambda^{\mathrm{pr}}\in \mbox{Pic}\bigl(\rr_{1+\frac{3}{2}g(g-1)}\bigr)$ be the Hodge-Prym class and  $\delta_0', \delta_0''$ and $\delta_0^{\mathrm{ram}}$ be the boundary classes on $\rr_{1+\frac{3}{2}g(g-1)}$ corresponding to Prym structures having an irreducible underlying curve, see \cite[Example 1.4]{FL} for details. In particular $\delta_0^{\mathrm{ram}}$ is the class of the ramification divisor $\Delta_0^{\mathrm{ram}}$ of the  morphism $\rr_{1+\frac{3}{2}g(g-1)}\rightarrow \mm_{1+\frac{3}{2}g(g-1)}$ and its general point corresponds to an admissible double cover $f\colon Y'\rightarrow Y$, where $Y$ is an irreducible stable curve of genus $1+\frac{3}{2}g(g-1)$ having a single node $n\in Y$, whereas $Y'$ is an irreducible stable curve of arithmetic genus $1+3g(g-1)$ having a single node $n'\in Y'$ and such that $f^{-1}(n)=\{n'\}$.   We write

\begin{equation}\label{eq:szegohodge}
\lambda_{\mathrm{SzH}}:=\mathfrak{Sc}^*(\lambda')=c_{\lambda}\cdot \lambda-\sum_{i=0}^{\left\lfloor \frac{g}{2}\right\rfloor} \bigl(c_{\alpha_i}\cdot \alpha_i+ c_{\beta_i}\cdot \beta_i) \in \mathrm{Pic}(\ss_g^+).
\end{equation}

In this section we determine the coefficient $c_{\lambda}$ in this expression. We denote by $\cS_g^{\circ}$ the open substack of $\cS_g^+$ consisting of smooth spin curves $[C, \eta]$ such that $h^0(C, \eta)\leq 2$. It is well known, see \cite[Theorems 2.13, 2.18]{T1} that $\mbox{codim}\bigl(\cS_g^+\setminus \cS_g^{\circ}, \cS_g^+)\geq 5$, therefore replacing $\cS_g^+$ by $\cS_g^{\circ}$ has no effect on any codimension one calculation on $\cS_g^+$. Let $\rho\colon \cC_g^{\circ}\rightarrow \cS_g^{\circ}$ be the (restriction of the) universal spin curve and by $\eta_g\in \mbox{Pic}(\cC_g^{\circ})$ the (restriction of the) universal spin bundle. Let $\mathbb E$ be the Hodge bundle on $\cS_g$ and we also introduce the rank $2g-2$ vector bundle on $\cS_g$

\begin{equation}\label{eq:szh}
\mathbb E_{3}:=\rho_*\bigl(\omega_{\rho}\otimes \eta_g\bigr), \ \   \       \mathbb E_{3 \bigl| [C, \eta]}\cong H^0(C, \omega_C\otimes \eta)\cong H^0(C, \eta^{\otimes 3}), \ \mbox{ for any } [C, \eta]\in \cS_g.
\end{equation}

Denoting by $\mathbb E'$ the Hodge bundle on $\mm_{1+3g(g-1)}$ and by $\mathfrak{Sc}^{\circ}\colon \mathcal{S}_g^{\circ}\dashrightarrow \mm_{1+3g(g-1)}$ the restriction of the Scorza map $\mathfrak{Sc}$, the Szeg\H{o}-Hodge bundle is then, by definition, the pull-back
$$\mathbb E_{\mathrm{SzH}}=\bigl(\mathfrak{Sc}^{\circ}\bigr)^*(\mathbb E').$$
Slightly abusing the notation, in this proof we consider the restriction $\xi \colon \mathcal{S}_g^{\circ}\rightarrow \rr_{1+\frac{3}{2}g(g-1)}$ of the morphism $\xi$ to the open substack $\mathcal{S}_g^{\circ}$. We shall explicitly describe $\mathbb E_{\mathrm{SzH}}$ in terms of a tautological bundle on $\mathcal{S}_g^{\circ}$ defined in terms of the geometry of the symmetric square of curves.

We denote by $\mathbb E^{\mathrm{pr},+}$ (respectively by $\mathbb E^{\mathrm{pr},-}$) the invariant (respectively anti-invariant) part of the Hodge bundle on $\rr_{1+\frac{3}{2}g(g-1)}$. Using \cite[Proposition 4.1]{FL}, we obtain

\begin{equation}\label{eq:antiinv}
c_1(\mathbb E^{\mathrm{pr},+})=\lambda^{\mathrm{pr}},  \ \ \mbox{ } \ \ c_1(\mathbb E^{\mathrm{pr},-})=\lambda^{\mathrm{pr}}-\frac{\delta_0^{\mathrm{ram}}}{4} \ \ \mbox{ and } \ \ \chi^*(\lambda')=2\lambda^{\mathrm{pr}}-\frac{\delta_0^{\mathrm{ram}}}{4}.
\end{equation}

Furthermore, we introduce the invariant (respectively anti-invariant) part of the Szeg\H{o}-Hodge bundle, by setting
\begin{equation}\label{eq:inv_szeg}
\mathbb E_{\mathrm{SzH}}^{+}:=\xi^*\bigl(\mathbb E^{\mathrm{pr}, +}\bigr) \ \ \mbox{ and } \ \ \mathbb E_{\mathrm{SzH}}^{-}:=\xi^*\bigl(\mathbb E^{\mathrm{pr}, -}\bigr)
\end{equation}.

We shall separately determine the first Chern classes of the bundles $\mathbb E_{\mathrm{SzH}}^{+}$ and $\mathbb E_{\mathrm{SzH}}^{-}$.

\begin{proposition}\label{prop:inszeg}
One has $c_1\bigl(\mathbb E_{\mathrm{SzH}}^{+}\bigr)=\frac{20g-7}{2}\lambda\in \mathrm{Pic}(\ss_g^+)$.
\end{proposition}

\begin{proof}
We start with a point $[C,\eta]\in \cS_g^{\circ}$. Via Definition \ref{def:limitscorza}, we introduced the double cover $S(C,\eta)\rightarrow T(C,\eta)$, where $T(C,\eta)$ is the only curve of the linear system $\bigl|\cL_{\eta}(\delta)\bigr|$ on $C^{(2)}$. We have the canonical identification of the fibre of $\mathbb E_{\mathrm{SzH}}^{+}$ over a point $[C, \eta]\in \cS_g^{\circ}$

$$\mathbb E^{+}_{\mathrm{SzH} \bigl|[C, \eta]} \cong H^0\bigl(T(C,\eta), \omega_{T(C,\eta)}\bigr).$$

\vskip 3pt

By the adjunction formula
$$\omega_{T(C,\eta)}\cong \omega_{C^{(2)}} \bigl(T(C,\eta))_{\bigl |T(C,\eta)}\cong \cL_{\omega_C\otimes \eta}(-\delta+\delta)_{\bigl | T(C, \eta)} \cong \cL_{\eta^{\otimes 3}\bigl | T(C,\eta)}.$$
Twisting by $\cL_{\eta^{\otimes 3}}$ the exact sequence

\begin{equation}\label{eq:ex_seq2}
0\longrightarrow \cO_{C^{(2)}}(-T(C,\eta))\lra \cO_{C^{(2)}}\lra \cO_{T(C,\eta)}\lra 0,
\end{equation}
then taking cohomology, we obtain the following long exact sequence:
\begin{align*}
0\lra H^0\bigl(C^{(2)}, \cL_{\omega_C}(-\delta)\bigr)\lra H^0\bigl(C^{(2)}, \cL_{\eta^{\otimes 3}}\bigr)\lra H^0\bigl(T(C,\eta),\omega_{T(C,\eta)}\bigr)\\
\lra H^1\bigl(C^{(2)}, \cL_{\omega_C}(-\delta)\bigr)\lra H^1\bigl(C^{(2)}, \cL_{\eta^{\otimes 3}}\bigr)\lra 0.
\end{align*}
Using the  identifications (\ref{eq:cohsympr}), we obtain the following natural exact sequence:
$$0\longrightarrow \bigwedge^2 H^0(\omega_C)\lra \mbox{Sym}^2 H^0(\omega_C\otimes \eta)\lra H^0(\omega_{T(C,\eta)})\lra H^1(\omega_C)\otimes H^0(\omega_C)\lra 0,$$
which globalizes to the following exact sequence of vector bundles on $\cS_g^{\circ}$:
\begin{equation}\label{eq:exseqsz}
0\lra \bigwedge^2 \mathbb E\lra \mbox{Sym}^2 \mathbb E_{3} \lra \mathbb E_{\mathrm{SzH}}^{+} \lra \mathbb E\lra 0.
\end{equation}

In order to estimate the first Chern class of $\mathbb E_{3}$ (which was defined via (\ref{eq:szh})), we apply Grothendieck-Riemann-Roch to the universal spin curve $\rho\colon \cC_g^{\circ}\rightarrow \cS_g^{\circ}$ and write

\begin{align*}
c_1\bigl(\mathbb E_{3}\bigr)=c_1\bigl(\rho_{!}(\omega_{\rho}\otimes \eta_g)\bigr)=\rho_*\Bigl[\Bigl(1+c_1(\omega_{\rho})+c_1(\eta_g)+\frac{(c_1(\omega_{\rho})+c_1(\eta_g))^2}{2}\Bigr)\cdot \\
\Bigl(1-\frac{c_1(\omega_{\rho})}{2}+\frac{c_1^2(\omega_{\rho})}{12}\Bigr)\Bigl]_2=\frac{11}{24}\rho_*\bigl(c_1^2(\omega_{\rho})\bigr)=
\frac{11}{24}\rho_*\bigl(c_1^2(\omega_{\rho})\bigr)=\frac{11}{2}\lambda,
\end{align*}
where we have used that $2c_1(\eta_g)=c_1(\omega_{\rho})$, as well as Mumford's formula $\rho_*\bigl(c_1^2(\omega_{\rho})\bigr)=12\lambda$.
From the sequence (\ref{eq:exseqsz}), using that $c_1\bigl(\mbox{Sym}^2 \mathbb E_{3}\bigr)=(2g-1)c_1\bigl(\mathbb E_{3}\bigr)$, we then compute that
$$
\xi^*(\lambda^{\mathrm{pr}})=c_1\bigl(\mathbb E_{\mathrm{SzH}}^{+}\bigr)=\lambda-c_1\bigl(\bigwedge^2 \mathbb E\bigr)+c_1\bigl(\mathrm{Sym}^2 \mathbb E_{3}\bigr)\\
=\lambda-(g-1)\lambda+(2g-1)\frac{11}{2}\lambda=\frac{20g-7}{2} \lambda.
$$
This completes the proof.
\end{proof}

We can now determine the weight of the Szeg\H{o}-Hodge bundle on $\ss_g^+$, that is, the $\lambda$-coefficient $c_{\lambda}$ in the expression (\ref{eq:szegohodge}) of the class $\lambda_{\mathrm{SzH}}$.

\begin{theorem}\label{eq:lambdaszego}
Sections of the Szeg\H{o}-Hodge bundle are Teichm\"uller modular forms of weight $\frac{77g-25}{4}$ on $\cS_g^+$, that is, $c_{\lambda}=\frac{77g-25}{4}$.
\end{theorem}

\begin{proof}
We are going to compute the first Chern class of the vector bundle $\xi^*\bigl(\mathbb E^{\mathrm{pr},-}\bigr)$, which, coupled with Proposition \ref{prop:inszeg}, via (\ref{eq:antiinv}), amounts to determining the class $\xi^*(\delta^{\mathrm{ram}})$.
The essential part in the proof is establishing that

\begin{equation}\label{eq:pull_back_thet}
\xi^*\bigl(\delta_0^{\mathrm{ram}}\bigr)=12(g-1)[\Theta_{\mathrm{null}}]\in \mbox{Pic}(\mathcal{S}_g^{\circ}).
\end{equation}

To that end, we first note that there is a set-theoretic equality $\xi^*(\Delta_0^{\mathrm{ram}})=\Theta_{\mathrm{null}}$. Indeed, given a spin curve $[C, \eta]\in \cS_g^{\circ}$, assuming that $\xi([C,\eta]) \in \Delta_0^{\mathrm{ram}}$, necessarily the associated cover $S(C,\eta) \rightarrow T(C,\eta)$ is ramified, in particular $S(C,\eta)$ intersects the diagonal $\Delta$, therefore we obtain $h^0(C,\eta)\geq 2$. The converse inclusion has been established in Theorem \ref{propthetanull2limit}. To complete the proof of the claim, we  explicitly describe the pull-back $\xi^*\bigl(\mathbb E^{\mathrm{pr},-})$.

\vskip 4pt

We denote by $Z$ the subvariety of the universal symmetric product $\rho_2 :\bigl(\mathcal{C}_g^{\circ}\bigr)^{(2)}\ra \cS^{\circ}_g$ consisting of the triples
$[C, \eta,  2\cdot x]$ such that $h^0(C,\eta)=2$ and  $H^0(C,\eta(-2x))\neq 0$. Furthermore, we let $\mathcal{T}$ be the effective divisor on $\bigl(\mathcal{C}_g^{\circ}\bigr)^{(2)}$ whose fiber at every point $[C,\eta]\in \cS_g^{\circ}$ is $T(C,\eta)$. Clearly $Z\subseteq \mathcal{T}$ and $\mbox{codim}(Z, \mathcal{T})=2$. By slight abuse of notation we denote by $\rho_2=\rho_{2\bigl| \mathcal{T}}\colon \mathcal{T}\rightarrow \cS_g^{\circ}$ the universal curve and $\delta_g :=\delta_{g\bigl|\mathcal{T}}\in \mbox{Pic}(\mathcal{T})$. We consider the following vector bundle on $\cS_g^{\circ}$
\begin{equation}\label{eq:B}
\mathcal{B}:=(\rho_2)_*\bigl(\omega_{\rho_2}\otimes \delta_{g}\bigr), \ \ \mbox{with fibres } \cB_{\bigl |[C, \eta]}=H^0\bigl(C^{(2)}, \omega_{T(C,\eta)}\otimes \delta_{|T(C,\eta)}\bigr).
\end{equation}
The fact that $\cB$ is locally free and has this description of its fibres follows from Grauert's theorem, coupled with Proposition \ref{lemeta02}. Let $f\colon \widetilde{\mathcal{T}}\rightarrow \mathcal{T}$ be the blow-up of the subscheme $Z\subseteq \mathcal{T}$ and introduce the universal curve $$\widetilde{\rho}\colon \widetilde{\mathcal{T}}\rightarrow \cS_g^{\circ}, \ \ \  \widetilde{\rho}=\rho_2\circ f.$$
We denote by $\cE$ the exceptional divisor of $f$.

\vskip 3pt

We fix a general point $[C, \eta]\in \Theta_{\mathrm{null}}$ as in Theorem \ref{propthetanull2limit}. In particular, we write again $\mathfrak{Ram}=x_1+\cdots+x_{4g-4}$ for the ramification divisor of the pencil $|\eta|$. Considering the curve $\Gamma_{\eta}$ introduced in (\ref{eq:trace_curve}), we denote by $\tilde{x}_j:=2\cdot x_j\in C^{(2)}$ the points of intersection of $\Gamma_{\eta}$ with the diagonal $\overline{\Delta}$. Identifying $\overline{\Delta}$ with $C$, the image $\xi\bigl([C,\eta]\bigr)$ is then the stable Prym curve
$$\bigl[X=\Gamma_{\eta} \cup E_1 \cup \ldots \cup E_{4g-4} \cup C, \ \epsilon\bigr] \in \rr_{1+\frac{3}{2}g(g-1)},$$
where each $E_j$ is a smooth rational curve meeting $\Gamma_{\eta}$ at $\tilde{x}_j$ and $C$ at $x_j$ for $j=1, \ldots, 4g-4$ (while being disjoint from the other components of $X$), whereas $\epsilon \in \mbox{Pic}^0(X)$ is the line bundle such that $\epsilon_{|E_j}=\cO_{E_j}(1)$ for all $j$,  $\epsilon_{|\Gamma_{\eta}}=\delta^{\vee}_{|\Gamma_{\eta}}\in \mbox{Pic}^{2-2g}(\Gamma_{\eta})$ and $\epsilon_{|C}=\omega_C^{\vee}$.
We have the natural fibrewise identification $\xi^*(\mathbb E^{\mathrm{pr},-})_{\bigl |[C,\eta]}\cong H^0(X, \omega_X\otimes \epsilon)$. Since $\omega_X\otimes \epsilon_{|E_j}\cong \cO_{E_j}(1)$ for all $j$, the evaluation map $\mathrm{ev}_{x_j, \tilde{x}_j}\colon H^0\bigl(E_j, \omega_X\otimes \epsilon_{E_j}\bigr)\rightarrow \mathbb C^2_{x_j, \tilde{x}_j}$ of the Mayer-Vietoris sequence on $X$ is an isomorphism. Therefore, from the Mayer-Vietoris sequence on $X$, we obtain the following canonical identification:
\begin{equation}\label{eq:can_X}
H^0(X,\omega_X\otimes \epsilon) \cong H^0 \bigl(C, \omega_C^{\otimes 2}\bigr) \oplus H^0\bigl(\Gamma_{\eta}, \omega_{\Gamma_{\eta}}\otimes \delta_{\Gamma_{\eta}}\bigr).
\end{equation}
On the other hand, writing down the Mayer-Vietoris sequence on the curve $T(C,\eta) = \Gamma_{\eta}\cup \overline{\Delta}$, taking into account that, via the identification $C\cong \overline{\Delta}$, we have $\cL_{\eta^{\otimes 3}}\otimes \delta_{\bigl |\overline{\Delta}}\cong \omega_C^{\otimes 2}$, we obtain the following exact sequence:

\begin{equation}\label{eq:can_T}
0\longrightarrow H^0\bigl(\omega_{T(C, \eta)}\otimes \delta\bigr) \longrightarrow H^0(C, \omega_C^{\otimes 2})\oplus H^0\Bigl(\Gamma_{\eta}, \omega_{\Gamma_{\eta}}\otimes \delta_{\Gamma_{\eta}}\bigl(\sum_{j=1}^{4g-4} \tilde{x}_j\bigr)\Bigr)\twoheadrightarrow \bigoplus_{j=1}^{4g-4} \mathbb C_{\tilde{x}_i}.
\end{equation}

\vskip 3pt

The pull-back $\rho_2^{-1}(\Theta_{\mathrm{null}})$ splits into two components $\mathfrak{D}_{\Gamma}$ and $\mathfrak{D}_{\overline{\Delta}}$ corresponding to the decomposition $T(C,\eta)=\Gamma_{\eta}\cup \overline{\Delta}$ given by Theorem \ref{propthetanull2limit}, for a general element $[C, \eta]\in \Theta_{\mathrm{null}}$. Note that $\mathfrak{D}_{\Gamma}\cap \mathfrak{D}_{\overline{\Delta}}=Z$.  We denote by $\widetilde{\mathfrak{D}}_{\Gamma}$, respectively, by $\widetilde{\mathfrak{D}}_{\overline{\Delta}}$, the strict transform in $\widetilde{\mathcal{T}}$ of  $\mathfrak{D}_{\Gamma}$, respectively, $\mathfrak{D}_{\overline{\Delta}}$. We have the following relations:
\begin{equation}\label{eq:push_fordiv}
\widetilde{\rho}_*\Bigl(\bigl[\widetilde{\mathfrak{D}}_{\Gamma}\bigr]\cdot \bigl[\widetilde{\mathfrak{D}}_{\overline{\Delta}}\bigr]\Bigr)=0, \ \  \widetilde{\rho}_*\Bigl(\bigl[\cE\bigr]^2\Bigr)=-8(g-1)\bigl[\Theta_{\mathrm{null}}\bigr], \ \  \widetilde{\rho}_*\Bigl(\bigl[\widetilde{\mathfrak{D}}_{\Gamma}\bigr]^2\Bigr)=-(4g-4)\bigl[\Theta_{\mathrm{null}}].
\end{equation}
A local analysis relying on (\ref{eq:can_X}) and on (\ref{eq:push_fordiv}) shows that one has the following canonical identifications of vector bundles on $\cS_g^{\circ}$
$$\xi^*\bigl(\mathbb E^{\mathrm{pr}, -}\bigr)\cong \widetilde{\rho}_*\Bigl(\omega_{\widetilde{\rho}}\otimes f^*(\delta_g)\bigl(\widetilde{\mathfrak{D}}_{\Gamma}\bigr)\Bigr) \ \ \mbox{ and } \ \ \cB\cong \widetilde{\rho}_*\Bigl(\omega_{\widetilde{\rho}}\otimes f^*(\delta_g)\bigl(\widetilde{\mathfrak{D}}_{\Gamma}+\mathcal{E}+\widetilde{\mathfrak{D}}_{\overline{\Delta}}\bigr)\Bigr).$$

From the exact sequence
$$0\lra \omega_{\widetilde{\rho}}\otimes f^*(\delta_g)\lra \omega_{\widetilde{\rho}}\otimes f^*(\delta_g)\bigl(\widetilde{\mathfrak{D}}_{\overline{\Delta}}+\mathcal{E}\bigr)\lra \omega_{\widetilde{\rho}}\otimes f^*(\delta_g)\otimes \cO_{\widetilde{\mathfrak{D}}_{\overline{\Delta}}+\mathcal{E}}\bigl(\widetilde{\mathfrak{D}}_{\overline{\Delta}}+\mathcal{E}\bigr)\lra 0 $$
on $\widetilde{\mathcal{T}}$, we obtain, after tensoring with $\mathcal{O}_{\widetilde{\mathcal{T}}}(\widetilde{\mathfrak{D}}_{\Gamma})$ and pushing forward under $\widetilde{\rho}$, the following exact sequence
\begin{equation}\label{eq:eval}
0\lra \xi^*\bigl(\mathbb E^{\mathrm{pr}, -}\bigr)\lra \cB \stackrel{\mathrm{ev}}\lra \widetilde{\rho}_*\Bigl(\omega_{\widetilde{\rho}}\otimes f^*(\delta_g)\otimes \cO_{\widetilde{\mathfrak{D}}_{\overline{\Delta}}+
\mathcal{E}}(\widetilde{\mathfrak{D}}_{\Gamma}+\mathcal{E}+\widetilde{\mathfrak{D}}_{\overline{\Delta}})\Bigr).
\end{equation}
The image of the morphism $\mathrm{ev}$ is then a vector bundle of rank $3g-3$ supported on the divisor $\Theta_{\mathrm{null}}$, corresponding fiberwise to the image of the map $H^0\bigl(C, \omega_C^{\otimes 2}\bigr)\rightarrow \bigoplus_{j=1}^{4g-4} \mathbb C_{\tilde{x}_j}$ described in (\ref{eq:can_T}). Hence, from (\ref{eq:eval}) we conclude  that

\begin{equation}\label{eq:ChernB1}
c_1\bigl(\xi^*(\mathbb E^{\mathrm{pr}, -})\bigr)=c_1(\cB)-(3g-3)[\Theta_{\mathrm{null}}].
\end{equation}
It remains to determine the class $c_1(\cB)$, which is a relatively straightforward task along the lines of Proposition \ref{prop:inszeg}.

We tensor the exact sequence (\ref{eq:ex_seq2}) on $C^{(2)}$ by $\cL_{\eta^{\otimes 3}}(\delta)$, take cohomology and then use the identifications (\ref{eq:cohsympr}), in order to obtain the following long exact sequence:
$$
0\longrightarrow \mbox{Sym}^2 H^0(\omega_C)\lra \mbox{Sym}^2 H^0\bigl(\cL_{\eta^{\otimes 3}}(\delta)\bigr) \lra H^0\bigl(\omega_{T(C,\eta)}\otimes \delta\bigr) \lra H^1(\omega_C)\otimes H^0(\omega_C)\lra 0.
$$
One also has the exact sequence, obtained by tensoring by $\cL_{\eta^{\otimes 3}}(\delta)$ and taking cohomology in the long exact sequence
$0\rightarrow \cO_{C^{(2)}}(-2\delta) \rightarrow \cO_{C^{(2)}} \rightarrow \cO_{\overline{\Delta}} \rightarrow 0$:
$$0\longrightarrow H^0(\cL_{\eta^{\otimes 3}}(-\delta))\lra H^0(\cL_{\eta^{\otimes 3}}(\delta))\lra H^0(\omega_C^{\otimes 2})\lra H^1(\cL_{\eta^{\otimes 3}}(-\delta))\lra \cdots.$$
Observe that via the identifications (\ref{eq:cohsympr}), one has $H^1(C^{(2)}, \cL_{\eta^{\otimes 3}}(-\delta))=0$. Recalling the notation (\ref{eq:univ_spin3}) for the line bundle $\cL_{\eta_g}$ on the universal symmetric product $\mathcal{C}_g^{(2)}$, we introduce the following vector bundle on $\cS_g^{\circ}$
\begin{equation}\label{eq:F}
\mathbb E_{3|1}:=(\rho_2)_*\bigl(\cL_{\eta_g}^{\otimes 3}(\delta_g)\bigr), \ \mbox{ with fibres } \mathbb E_{3|1\ \bigl|[C, \eta]}\cong H^0\bigl(C^{(2)}, \cL_{\eta^{\otimes 3}}\otimes \delta\bigr) \mbox{ over } [C,\eta]\in \cS_g^{\circ}.
\end{equation}
The above sequences being natural, they induce the respective exact sequences of vector bundles on $\cS_g^{\circ}$
$$0\lra \mbox{Sym}^2 \mathbb E\lra \mathbb E_{3|1}\lra \cB\lra \mathbb E\lra 0,  \ \ $$
\begin{equation}\label{eq:obstruction}
0\lra \bigwedge^2 \mathbb E_{3}\lra \mathbb E_{3|1}\lra \rho_*(\omega_{\rho}^{\otimes 2})\lra 0.
\end{equation}
Using Mumford's formula \cite[Theorem 5.10]{M1}  $c_1\bigl(\rho_*(\omega_{\rho}^{\otimes 2})\bigr)=13\lambda\in \mbox{Pic}(\cS_g^{\circ})$, we compute

$$
c_1(\cB)=c_1\bigl(\mathbb E_{3|1}\bigr)-c_1\bigl(\mbox{Sym}^2 \mathbb E\bigr)+\lambda=13\lambda+c_1\bigl(\bigwedge^2 \mathbb E_{3})+\lambda-c_1\bigl(\mbox{Sym}^2 \mathbb E\bigr)=\frac{20g-7}{2}\lambda.$$

We can now conclude, using (\ref{eq:ChernB1}) and (\ref{eq:antiinv}),
\begin{align*}
c_1\bigl((\mathfrak{Sc}^{\circ})^*(\lambda')\bigr)=c_1\bigl(\xi^*(\mathbb E^{\mathrm{pr},+})\bigr)+c_1\bigl(\xi^*(\mathbb E^{\mathrm{pr},-})\bigr)=\frac{20g-7}{2}\lambda+c_1(\cB)-(3g-3)[\Theta_{\mathrm{null}}]\\
=\frac{20g-7}{2}\lambda+\frac{20g-7}{2}\lambda-\frac{3(g-1)}{4}\lambda=\frac{77g-25}{4}\lambda.
\end{align*}

\end{proof}

\begin{remark}\label{rmk:non_proj}
As explained in \cite[Proposition 3.1]{DW}, the vector bundle $\mathbb E_{3|1}$ described by (\ref{eq:F}) is the \emph{obstruction bundle} determining whether the moduli space $\mathfrak{M}_g$ of supersymmetric curves is split, that is, there exists a projection $\mathfrak{M}_g\rightarrow \cS_g^+$. It is the non-triviality of the extension class
$$\bigl[\mathbb E_{3|1}\bigr]\in \mbox{Ext}^1\Bigl(\rho_*(\omega_{\rho}^{\otimes 2}\bigr), \bigwedge^2 \mathbb E_3\Bigr)$$
induced by the exact sequence (\ref{eq:obstruction}) that is responsible for the moduli space $\mathfrak{M}_g$ being non-split for $g\geq 5$.
\end{remark}

\section{The Szeg\H{o}-Hodge class on the boundary of $\ss_g^+$}

In order to determine the boundary coefficients of the Szeg\H{o}-Hodge class $\lambda_{\mathrm{SzH}}$  in (\ref{eq:szegohodge}), we  calculate the intersection of $\lambda_{\mathrm{SzH}}$ with  standard test curves lying in the boundary of $\ss_g^+$.

\subsection{Test curves in $\ss_g^+$.}

For each  $2\leq i\leq g-1$, let us fix general curves $[C]\in \cM_{i}$ and $[D, q]\in
\cM_{g-i, 1}$ and consider the following test
curves $F_i\subseteq A_i$ and $G_i\subseteq B_i$ inside $\ss_g^+$. We fix even (respectively, odd)
theta-characteristics $\eta_C^+\in \mbox{Pic}^{i-1}(C)$ and
$\eta_D^+\in \mbox{Pic}^{g-i-1}(D)$ (respectively, $\eta_C^-\in
\mathrm{Pic}^{i-1}(C)$ and $\eta_D^-\in \mathrm{Pic}^{g-i-1}(D)$).

\vskip 4pt

We define the family
$F_i$ (respectively $G_i$) as consisting of the following stable spin curves
\[
F_i:=\Bigl\{t:=[C\cup_{p=q} D,\ \ \eta_C^+, \ \eta_D^+]\in \ss_g^{+}: p\in
C\Bigr\}
\]
respectively
\[
G_i:=\Bigl\{t:=[C\cup_{p=q} D,\  \eta_C^-, \ \eta_D^-]\in \ss_g^{+}: p\in C\Bigr\}.
\]
Then we have the following formulas, see also \cite{F2}:
\begin{equation}\label{eq:intFGi}
F_i \cdot \lambda = F_i\cdot \beta_i=0, \ \ F_i\cdot \alpha_i=2-2i, \ \mbox{ and } \  G_i\cdot \lambda =G_i\cdot \alpha_i=0,  \ G_i\cdot \beta_i=2-2i.
\end{equation}
The intersection numbers of $F_i$ and $G_i$ with the remaining generators of $\mbox{Pic}(\ss_g^+)$ equal zero.

\vskip 4pt

The next one-parameter family of spin curves we construct provides a covering curve for the boundary divisor $B_0\subseteq \ss_g^+$. Let us fix a general point $[C, p, \eta_C]\in \cS_{g-1, 1}^+$. Setting $E\cong \PP^1$
for an exceptional component, we define
$$H_0:=\Bigl\{[C\cup_{\{p, q\}} E, \ \eta_C, \ \eta_E=\mathcal{O}_E(1)]:
q\in C\Bigr\}\subseteq \ss_g^+.$$ The fibre of $H_0$ over the point $q=p\in
C$ is the even spin curve $$\bigl[C\cup_q E'\cup _{p'}
E''\cup_{\{p'', q''\}} E, \ \eta_C,\
\eta_{E'}=\cO_{E'}(1), \ \eta_E=\cO_E(1),
\eta_{E''}=\cO_{E''}(-1)\bigr],$$ having as stable model
$[C\cup _p E_{\infty}]$, where $E_{\infty}$ is the rational nodal curve corresponding to $j=\infty$.
Here $E', E''$ are rational curves, $E'\cap E''=\{p'\}$,
 $E\cap E''=\{p'', q''\}$ and the stabilization map for $C\cup E\cup E'\cup E''$ contracts the components $E'$ and $E$.

\vskip 3pt

We find

\begin{equation}\label{eq:H0}
H_0 \cdot \alpha_1 = 1, \ H_0\cdot \lambda= 0, \ H_0 \cdot \alpha_0 = H_0 \cdot \beta_1 = H_0\cdot \alpha_i=H_0\cdot
\beta_i=0,
\end{equation}
for $2\leq i\leq \bigl\lfloor \frac{g}{2}\bigr\rfloor$. Hence $H_0\cdot \beta_0=\frac{1}{2} \pi_*(H_0)\cdot \delta_0=1-g$.

\vskip 4pt

The last  test curve we introduce covers the divisor $A_1$. We fix a general pointed curve $[C,p]\in \cM_{g-1,1}$ and an
even theta characteristic $\eta_C=\eta_C^{+}$ on $C$. Let $\bigl\{[J_t, p_t]\bigr\}_{t\in \PP^1}\rightarrow \mm_{1,1}$ be a general pencil of plane cubics and we consider the family of spin curves

$$G_0:=\Bigl\{[C\cup_{p\sim p_t} J_t, \ \eta_C, \ \eta_t]: \eta_t\in \mathrm{Pic}^0(J_t)[2]\setminus \{\cO_{J_t}\}, \ t\in \PP^1\Bigr\}\subseteq \ss_g^+.$$

The intersection numbers of $G_0$ with the generators of $\mbox{Pic}(\ss_g^+)$ have been computed \cite{F2}:

\begin{equation}\label{eq:intpencilA1}
G_0\cdot \lambda=3, \  \ G_0\cdot \beta_0=G_0\cdot \beta_1=0, \ \ G_0\cdot \alpha_1=-3, \ \ G_0\cdot \alpha_0=12, \ \ G_0\cdot \beta_0=12.
\end{equation}
The intersection numbers of $G_0$ with the remaining boundary divisors $\alpha_j$ and $\beta_j$, where $2\leq j\leq \bigl\lfloor \frac{g}{2}\rfloor$, are manifestly equal to zero.

\subsection{The Szeg\H{o}-Hodge class along test curves} Using the precise description of the Scorza curve along the boundary divisors of $\ss_g^+$, we determine the intersection numbers of the test curves introduced above with the Szeg\H{o}-Hodge class.

\begin{proposition}\label{prop:inttest}
One has $G_0\cdot \lambda_{\mathrm{SzH}}=6g-3$, \ $H_0\cdot \lambda_{\mathrm{SzH}}=0$, as well as $F_i\cdot \lambda_{\mathrm{SzH}}=0$, for $i=2, \ldots, g-1$.
\end{proposition}
\begin{proof} We write $H_0\cdot \lambda_{\mathrm{SzH}}=H_0\cdot \mathfrak{Sc}^*(\lambda')=\bigl(\mathfrak{Sc}\bigr)_*(H_0)\cdot \lambda'$, where we have used that the rational Scorza map $\mathfrak{Sc}\colon \ss_g^+\dashrightarrow \mm_{1+3g(g-1)}$ is, as  explained in Proposition \ref{prop:B0}, regular along the image of the test curve $H_0$. To conclude that $(\mathfrak{Sc})_*(H_0)\cdot \lambda'=0$, we observe that as $t\in H_0$ varies, applying Proposition \ref{prop:B0}, the moduli of the components of the (stable) Scorza curve $\mathfrak{Sc}(t)$ \emph{do not} vary, only the points where these components are attached to each other. Similar consideration apply for the intersection number $F_i\cdot \lambda_{\mathrm{SzH}}=\bigl(\mathfrak{Sc}\bigr)_*(F_i)\cdot \lambda'$, where we can use Remark \ref{rem:Ai}, building directly on Proposition \ref{prop:Ai}.

\vskip 3pt

We are left with computing the intersection number $G_0\cdot \lambda_{\mathrm{SzH}}=\bigl(\mathfrak{Sc}\bigr)_*(G_0)\cdot \lambda'$. For each $[t, \eta_t]$, where $t\in \PP^1$, the curve $\mathfrak{Sc}(t)\in \mm_{1+3g(g-1)}$ consists of a copy of $C$, a copy of the genus one Scorza curve $\mathfrak{Sc}\bigl([J_t, \eta_t]\bigr)$, which is a translate by $\eta_t$ of the elliptic curve $J_t$, two copies of the fixed curve $C$ and, finally, $(2g-2)$ copies of $J_t$. The way these components meet each other is described in Remark \ref{rem:Ai}. In particular, as $[t, \eta_t]$ varies, $\mathfrak{Sc}(t)$ contains $2g-1$ copies of $J_t$ and these are the only components varying in moduli. The degree of the map $\ss_{1,1}^+\rightarrow \mm_{1,1}$  being equal to $3$, since the $\lambda$-degree of a pencil of plane cubics in $\mm_{1,1}$ equals $1$, we conclude that $\bigl(\mathfrak{Sc}\bigr)_*(G_0)\cdot \lambda'=3(2g-1)$, as claimed.
\end{proof}

\begin{proposition}\label{prop:coefficients}
In the notation of (\ref{eq:szegohodge}), we have:
$$c_{\alpha_0}=\frac{69g-21}{16}, \ \  c_{\beta_0}=0 \ \mbox{ and } c_{\alpha_i}=0, \mbox{ for } \ i=1, \ldots, \left\lfloor \frac{g}{2}\right\rfloor.$$
\end{proposition}
\begin{proof} This is a direct consequence of Theorems \ref{eq:lambdaszego} and Proposition \ref{prop:inttest}. Indeed, via (\ref{eq:intFGi}), for each $2\leq i\leq g-1$ we have that $0=F_i\cdot \lambda_{\mathrm{SzH}}=(2i-2)c_{\alpha_i}$, therefore $c_{\alpha_i}=0$, where we recall the convention $c_{\alpha_i}:=c_{\alpha_{g-i}}$ for $i > \left\lfloor \frac{g}{2}\right\rfloor$. Similarly, via (\ref{eq:H0}), we write
$0=H_0\cdot \lambda_{\mathrm{SzH}}=(g-1)c_{\beta_0}-c_{\alpha_1}$,
from which we conclude that $c_{\beta_0}=0$. Finally, using (\ref{eq:intpencilA1}) we write that
$$6g-3=G_0\cdot \lambda_{\mathrm{SzH}}=3c_{\lambda}-12c_{\alpha_0}-12c_{\beta_0},$$
and since $c_{\lambda}=\frac{77g-25}{4}$ (via Theorem \ref{eq:lambdaszego}), we obtain the claimed formula for $c_{\alpha_0}$. Finally, in order to conclude that the coefficient $c_{\beta_i}$ in (\ref{eq:szegohodge}) is non-negative for $i\geq 1$, it suffices to observe that
$(2i-2)c_{\beta_i}=G_i\cdot \mathfrak{Sc}^*(\lambda')=(\mathfrak{Sc})_*(G_i)\cdot \lambda'\geq 0$.
\end{proof}

\vskip 3pt

Recall that $\lambda', \delta_i'$ denote the standard generators of $\mbox{Pic}(\mm_{1+3g(g-1)})$. Before we describe the pull-back of the Scorza map $\mathfrak{Sc}$, we point out that there exists an effective divisor $\mathfrak{D}_{\mathrm{sg}}$ on $\cS_g^+$ consisting of ineffective spin curves $[C, \eta]$ such that the Scorza curve $S(C, \eta)$ is singular. Using \cite[Lemma 7.1.3]{DK}, this divisor has the following description
\begin{align*}
\mathfrak{D}_{\mathrm{sg}}:=\Bigl\{[C, \eta]\in \cS_g^+: H^0(C, \eta)=0, \ \exists (x,y)\in S(C,\eta) \mbox{ with } \ \\ H^0\bigl(C, \eta(x-2y)\bigr)\neq 0, \ \mbox{ and } \ H^0\bigl(C, \eta(y-2x)\bigr)\neq 0\Bigr\}^{-},
\end{align*}
where the closure is taken inside $\cS_g^+$. The task of computing the class $[\overline{\mathfrak{D}}_{\mathrm{sg}}]$ will be left for future work.

\begin{proposition}\label{prop:pull_back}
One has that $\mathfrak{Sc}^*(\delta_i)=0$, for $i=1, \ldots, \bigl \lfloor \frac{1+3g(g-1)}{2}\bigr \rfloor$. Furthermore
$$\mathfrak{Sc}^*(\delta_0')=(2g-1)\cdot \alpha_0+ (4g-2)\cdot \beta_0+\sum_{i=1}^{\lfloor \frac{g}{2}\rfloor} \Bigl(2(ig-i^2+g)\cdot \alpha_i+2i(g-i)\cdot \beta_i\Bigr) + \bigl[\overline{\mathfrak{D}}_{\mathrm{sg}}\bigr]+12(g-1)\bigl[\overline{\Theta}_{\mathrm{null}}\bigr].$$
\end{proposition}
\begin{proof}
The fact that $\mathfrak{Sc}^*(\delta'_i)=0$ for $i\geq 1$ follows from Theorem \ref{propthetanull2limit} and Propositions \ref{prop:limitA0}, \ref{prop:B0}, \ref{prop:Ai} and \ref{prop:limitBi}. Precisely, we use that the limit Scorza curve of the general point of each of the divisors $\Theta_{\mathrm{null}}$, or the boundary divisors $A_i, B_i$ has no disconnecting node, therefore it does not belong to the boundary divisor $\Delta_i'$ of $\mm_{1+3g(g-1)}$. The statement concerning $\mathfrak{Sc}^*(\delta_0')$ is clear set-theoretically. The multiplicities in front of each divisor correspond to the number of nodes of the limiting Scorza curve in question. This is contained in (\ref{eq:pull_back_thet}), concerning the multiplicity of $\bigl[\overline{\Theta}_{\mathrm{null}}\bigr]$, and in Remarks \ref{rem:Ai} and \ref{rem:Bi}, concerning the multiplicities of $\alpha_i$, $\beta_i$ respectively.
\end{proof}

\section{Scorza quartics via Wirtinger duality}

In this section we determine the connection between the Scorza quartic associated to an ineffective spin curve and theta functions on the Jacobian $JC$. The Scorza quartic in genus $3$ was introduced by Scorza \cite{Sc1} and rediscovered in \cite{DK} as the inverse map of the rational morphism $\cM_3\dashrightarrow \cS_3^{+}$, which assigns to a plane quartic its \emph{covariant quartic}. A generalization to arbitrary genus $g$ was put forward in \cite{Sc}. We shall provide a novel unconditional treatment of the Scorza quartic without using any classical invariant theory.

\vskip 4pt

We begin by recalling \emph{Wirtinger duality} for second order theta  functions. For $g\geq 3$, we fix a spin curve $[C, \eta]\in \cS_g^+$ with $H^0(C, \eta)=0$. We recall that $\Theta:=\Theta_{\eta}$ is the symmetric theta divisor on $JC=\mbox{Pic}^0(C)$ associated to $\eta$ (see also (\ref{eq:theta_eta})) and $\theta:=\theta_{\eta}$ is the symmetric theta function on $JC$ whose zero locus is $\Theta$.
For  $a\in JC$, let $\Theta_a:=\bigl\{\xi\in JC: H^0(C, \xi+\eta-a)\neq 0\bigr\}$ be the corresponding translated theta divisor.

\vskip 3pt

Wirtinger duality \cite[p. 335]{M3} is an isomorphism of linear systems
$$\mathfrak{w} \colon |2\Theta|\stackrel{\equiv}\longrightarrow |2\Theta|^{\vee},$$
such that the following diagram commutes:
$$
\xymatrix{
JC \ar[r]^{\psi} \ar[rd]_{\varphi_{2\Theta}}   &  |2\Theta| \ar[d]_{\mathfrak{w}}&\\
  &                       |2\Theta|^{\vee}.}
$$
Here $\psi\colon J\rightarrow |2\Theta|$  sends $a \in J$ to the divisor $\Theta_{a} + \Theta_{-a}$ and $\varphi_{2\Theta}\colon J\rightarrow \PP H^0\bigl(J, \cO_J(2\Theta)\bigr)^{\vee}$ is the map induced by the linear system $|2\Theta |$. Hence, via Wirtinger Duality, the image of the origin $0=\cO_C\in J$ corresponds to the divisor
$2\Theta \in |2\Theta|$.

\vskip 3pt

Let $|2\Theta|_0 \subseteq |2\Theta|$ be the sublinear system of divisors passing through $0 \in JC$. Then the map $\mathfrak{w}$  induces an isomorphism of $|2 \Theta|_0$ with the \emph{polar} of $2\Theta$, that is, the span $\langle 2\Theta \rangle$ of the divisor $2\Theta$ in the space $|2\Theta|^{\vee}\cong \PP^{2^g-1}$. Via the surjective restriction map
\[
H^0\bigl(JC, \cO_{JC}(2\Theta)\bigr) \lra H^0(\Theta , \cO_{\Theta}(2\Theta)\bigr),
\]
the span $\langle 2 \Theta \rangle \subseteq |2\Theta|^{\vee}$ can be identified with the subspace $\bigl|2\Theta_{|\Theta}\bigr|^{\vee}=
\PP H^0\bigl(\Theta, \cO_{\Theta}(2\Theta)\bigr)^{\vee}$ of hyperplanes in $|2\Theta|$ passing through the point $2\Theta$. So Wirtinger Duality also induces an isomorphism
$$
\mathfrak{w}\colon |2\Theta|_0 \stackrel{\cong}\lra \bigl|2\Theta_{|\Theta}\bigr|^{\vee}.
$$

Since we assumed that $H^0(C, \eta)=0$, we have  $0\not\in \Theta$, therefore the divisor $2\Theta$ does not belong to $|2\Theta|_0$ and the restriction map induces the isomorphism described in the Introduction in (\ref{eq:res1})
$$
\mathrm{res}\colon H^0\bigl(JC, 2\Theta\bigr)_0 \stackrel{\cong}\lra H^0\bigl(\Theta, \cO_{\Theta}(2\Theta)\bigr).
$$
Consequently, we also have an induced isomorphism of projective spaces
\begin{equation}\label{eq:tilde_wir}
\widetilde{\mathfrak{w}}=\mathfrak{w}\circ \mathrm{res}^{-1} \colon \bigl|2\Theta_{| \Theta}\bigr| \stackrel{\cong}\lra \bigl|2\Theta _{| \Theta}\bigr|^{\vee}.
\end{equation}

\vskip 4pt

Next, we observe that if $\varphi\colon C\times C\rightarrow JC$ is the difference map $(x,y)\mapsto \cO_C(x-y)$, then
$$\varphi^*\bigl(\cO_{JC}(2\Theta)\bigr)\cong \omega_C\boxtimes \omega_C(2\Delta),$$
and the pull-back on global sections induces the map $\varphi^*\colon H^0(JC, 2\Theta)_0\twoheadrightarrow \mbox{Sym}^2 H^0(C,\omega_C)$ described by  (\ref{eq:ident3}) in the Introduction. As explained in \cite[4.5]{welters86}, the map $\varphi^*$ assigns to a second order theta function vanishing at the origin the quadratic term of its Taylor expansion around that point.

\vskip 3pt

We now construct a natural map
\[
\mathrm{Sym}^2 H^1(JC, \cO_{JC}) \lra H^0\bigl(\Theta, \cO_{\Theta } (2\Theta )\bigr)
\]
defined as follows. From the cohomology of the short exact sequence
\[
0 \lra \cO_{JC} \lra \cO_{JC}(\Theta ) \lra \cO_{\Theta }(\Theta) \lra 0,
\]
since $H^1\bigl(JC, \cO_{JC}(\Theta)\bigr)=0$, we obtain the isomorphism $$H^1(JC, \cO_{JC})\cong H^0\bigl(\Theta, \cO_{\Theta }(\Theta)\bigr).$$ To make this identification explicit, we use \cite[0.3]{BD1}. The Abel-Jacobi map $C\rightarrow JC$ induces an isomorphism $H^1(JC, \cO_{JC})\stackrel{\cong}\rightarrow H^1(C, \cO_C)\cong T_{0}(JC)$. Since $\Theta$ is the vanishing locus of the function $\theta$ on $JC$, given a vector $v\in T_0(JC)$, the restriction to $\Theta$ of the function $d\theta.v=\partial_v(\theta_{|\Theta})$ can be regarded as an element of $H^0\bigl(\Theta, \cO_{\Theta}(\Theta)\bigr)$, whose zero locus   we denote by $\partial_{v} \Theta$.
The multiplication map of sections
$
\mbox{Sym}^2 H^0\bigl(\Theta, \cO_{\Theta}(\Theta)\bigr) \lra H^0\bigl(\Theta, \cO_{\Theta }(2\Theta)\bigr)
$
induces via this identification the natural map
\begin{equation}\label{eq:map_mu}
\mu \colon \mbox{Sym}^2 H^1(JC, \cO_{JC}) \lra H^0\bigl(\Theta, \cO_{\Theta}(2\Theta )\bigr).
\end{equation}

Our first result concerning this map is the following:

\begin{lemma}\label{lemma:map_bdry3}
One has the following commutative diagram of projectivized linear maps:
\begin{equation}\label{eq:wirt4}
\xymatrix{
|2\Theta|_0 \ar[r]^{\mathfrak{w}} \ar[rd]_{\varphi^*}   &  \bigl|2\Theta_{|\Theta}\bigr|^{\vee} \ar[d]^{^t\mu}&\\
  &                       \PP \ \mathrm{Sym}^2\ H^0(C, \omega_C)}
\end{equation}
\end{lemma}

\begin{proof}
We choose dual bases $(\del_1, \ldots, \del_g)$ and $(\omega_1, \ldots, \omega_g)$ of $H^1(C, \cO_C)\cong T_0(JC)$ and of $H^0(C, \omega_C)^{\vee}\cong  T_0(JC)^{\vee}$, respectively. Denoting by $(z_1, \ldots, z_g)$ the coordinates corresponding to the canonical flat structure on $JC$ giving the identification $T_0(JC)\cong \mathbb C^g$, we have  $$\mu(\del_i\cdot \del_j)=\del_i\theta\cdot \del_j \theta=\frac{\partial \theta}{\partial z_i}\cdot \frac{\partial \theta}{\partial z_j}\in H^0\bigl(\Theta, \cO_{\Theta}(2\Theta)\bigr).$$

Given a point $a\in \Theta$, Wirtinger Duality sends the divisor $\Theta_a + \Theta_{-a}$ to the hyperplane of divisors of $|2\Theta_{|\Theta}|$ which pass through $a$. The transpose map $^t\mu$ of the map $\mu$ introduced in (\ref{eq:map_mu})  maps this onto the hyperplane $H_a$ consisting of elements $\sum_{i\leq j} c_{ij} \del_i\cdot \del_j\in \mbox{Sym}^2 T_0(JC)$, such that
$$\sum_{i,j=1}^g c_{ij} \frac{\partial \theta}{\partial z_i}(a)\cdot \frac{\partial \theta}{\partial z_j}(a)=0.$$

The equation of the quadric tangent cone at the origin  to the divisor $\Theta_a + \Theta_{-a}$ equals  $\bigl(\sum_{i=1}^g \del_i \theta(a)\cdot \omega_i)^2\in \mbox{Sym}^2 H^0(C, \omega_C)$. This takes the value $\sum c_{ij} \del_i\theta (a)\cdot  \del_j \theta(a)$ on $\sum c_{ij} \del_i\cdot \del_j$, hence defines the hyperplane $H_a$, which proves the commutativity of the diagram.
\end{proof}

Keeping the notation of Lemma \ref{lemma:map_bdry3} and recalling that $(z_1, \ldots, z_g)$ are local coordinates on $T_0(JC)$ corresponding to a choice of basis of $H^1(C, \cO_C)$, we consider the Taylor expansion at the origin
$$
\theta=\theta_0+\theta_2+\theta_4+\cdots
$$
of the (even)   section $\theta=\theta(z_1, \ldots, z_g)$ of $\cO_{JC}(\Theta)$. Therefore,  $\theta_0=\theta(0)$,\ $\theta_2\in \mbox{Sym}^2 H^0(C,\omega_C)$ and $\theta_4\in  \mbox{Sym}^4 H^0(C, \omega_C)$, respectively. We have the following fundamental result:

\begin{theorem}\label{thm:thet_quartic}
The composition map
\[
\beta \colon \mathrm{Sym}^2 H^1(C, \cO_C ) \stackrel{\mu}{\lra} H^0\bigl(\Theta, \cO_{\Theta } (2\Theta )) \stackrel{\widetilde{\mathfrak{w}}}{\lra} H^0\bigl(\Theta, \cO_{\Theta}(2\Theta)\bigr)^{\vee} \stackrel{^t\mu}{\lra} \mathrm{Sym}^2 H^1(C, \cO_{JC})^{\vee}
\]
is fully symmetric, that is, it is induced by a quartic polynomial $F(C,\eta)$. Furthermore, this quartic equals
\[
F(C, \eta)= \frac{1}{2}\theta_2^2 -\theta_0\cdot \theta_4 \in \mathrm{Sym}^4 H^0(C, \omega_C).
\]
\end{theorem}

\begin{proof}
First observe that, via Lemma \ref{lemma:map_bdry3}, the map $\beta$ defined above coincides with the map in the statement of Theorem \ref{thm:scorza_quartic}.

\vskip 4pt

Let $F\in \mbox{Sym}^2 \bigl(\mbox{Sym}^2 H^0(C, \omega_C)\bigr)$ be the multilinear form representing $\beta$. We evaluate $F$ on a four-tuple of vectors $(v_1, v_2, v_3, v_4)$ from $T_0(JC)$.

The image of $(v_1, v_2)$ by the map $\mu$ is the section $\del_{v_1}\theta\cdot \del_{v_2}\theta\in H^0\bigl(\Theta, \cO_{\Theta}(2\Theta)\bigr)$, where $\del_{v_i}\theta\in H^0\bigl(\Theta, \cO_{\Theta}(\Theta)\bigr)$. We need to identify the element of $|2\Theta|_0$ whose restriction to the divisor $\Theta$ is the divisor of zeros of $\del_{v_1} \theta\cdot  \del_{v_2} \theta$.

\vskip 4pt

For a fixed $v\in T_0(JC)$, we denote by $\theta_v$ the function on $JC$ given by $u\mapsto \theta(u+v)$, Consider the section $G(u,v):=\theta (u+v)\theta (u-v)$ on $JC\times JC$, where $u,v\in T_0(JC)\cong \mathbb C^g$. For any fixed $v\in T_0(JC)$, this is a second order theta function on $JC$, whose zero locus is the divisor $\Theta_v + \Theta_{-v}\in |2\Theta|$. We denote by $\theta_{-v}\cdot \del_{v_1} \theta_v - \theta_v\cdot \del_{v_1} \theta_{-v}$ the partial derivative of $\theta (u+v)\theta (u-v)$ with respect to the variable $v$ in the direction of $v_1$. As a function of $u$, this is still a second order theta function on $JC$, see \cite{BD2} or \cite{Iz}. The second derivative
\[
\del_{v_1}\del_{v_2}\bigl(G\bigr)= \theta_{-v}\cdot \del_{v_1}\del_{v_2} \theta_v - \del_{v_1}\theta_v\cdot \del_{v_2} \theta_{-v} - \del_{v_1}\theta_{-v}\cdot \del_{v_2} \theta_v + \theta_v\cdot \del_{v_1}\del_{v_2} \theta_{-v}
\]
is also a second order theta function. At $v=0$, this second derivative, viewed as a function of $u$, is equal to
\[
\del_{v_1}\del_{v_2}\bigl(G\bigr)(u,0) = 2\theta\cdot \del_{v_1}\del_{v_2}\theta  -2\del_{v_1} \theta\cdot \del_{v_2} \theta,
\]
and, when restricted to $\Theta$, equals the section $-2\del_{v_1} \theta \cdot \del_{v_2} \theta\in H^0\bigl(\Theta, \cO_\Theta(2\Theta)\bigr)$. To obtain a section of $H^0(JC, 2\Theta)_0$ which restricts to $\del_{v_1} \theta\cdot \del_{v_2} \theta$, we  add the multiple $a\cdot \theta^2$ to $-\frac{1}{2} \del_{v_1}\del_{v_2}(G)(u,0)$, where $a=\frac{1}{\theta_0}\cdot \del_{v_1}\del_{v_2}\theta(0)$.

\vskip 4pt

To evaluate this element under $\phi^*$, we compute the equation of the quadric tangent at $0$ to the divisor of $-\frac{1}{2} \del_{v_1}\del_{v_2}(G)(u,0) + a\cdot \theta^2$. Choosing elements $v_3, v_4\in T_0(JC)$, we find that

$$\phi^*\bigl(a\cdot \theta^2\bigr)(v_3,v_4)=2a\theta_0\cdot \del_{v_3}\del_{v_4}\theta=2\del_{v_1}\del_{v_2}\theta\cdot \del_{v_3}\del_{v_4}\theta,$$
and
\begin{align*}
\phi^*\Bigl(-\frac{1}{2}\del_{v_1}\del_{v_2}G(u,0)\Bigr)(v_3,v_4)=-\theta_0\cdot \del_{v_1}\del_{v_2}\del_{v_3}\del_{v_4}\theta+\del_{v_1}\del_{v_3}\theta \cdot \del_{v_2}\del_{v_4}\theta+\\
\del_{v_1}\del_{v_4}\theta \cdot \del_{v_2}\del_{v_3}\theta- \del_{v_3}\del_{v_4}\theta\cdot \del_{v_1}\del_{v_2}\theta.
\end{align*} Combining these, we find

\begin{align*}
\beta \bigl(v_1, v_2, v_3, v_4\bigr) = \phi^*\Bigl(-\frac{1}{2}\del_{v_1}\del_{v_2}G(u,0)+a\cdot \theta^2\Bigr)(v_3, v_4)
= \del_{v_1}\del_{v_3} \theta \cdot \del_{v_2}\del_{v_4} \theta + \\ \del_{v_1}\del_{v_4} \theta \cdot \del_{v_2} \del_{v_3} \theta + \del_{v_3} \del_{v_4} \theta\cdot  \del_{v_1} \del_{v_2} \theta - \theta_0\cdot \del_{v_1} \del_{v_2} \del_{v_3} \del_{v_4} \theta =
 \del_{v_1} \del_{v_2} \del_{v_3} \del_{v_4} \left(\frac{1}{2} \theta_2^2 - \theta_0\cdot \theta_4 \right),
\end{align*}
which finishes the proof.
\end{proof}

\vskip 4pt

It remains to show that the quartic form $F(C, \eta)$ of Theorem \ref{thm:thet_quartic} is precisely the Scorza quartic considered in \cite{DK}. Recalling that $\phi_C\colon C\hookrightarrow \PP H^0(C, \omega_C)^{\vee}$ is the canonical embedding of $C$, for a point
$x\in C$, let $\phi_C(x)\in H^0 (\omega_C)^{\vee}\cong T_0(JC)$ be the image corresponding to evaluating canonical forms on $C$ at the point $x$.  Let $\del_x \Theta\in H^0\bigl(\Theta, \cO_{\Theta}(\Theta)\bigr)$ be the corresponding section under the isomorphism $H^1(C, \cO_C)\stackrel{\cong}\rightarrow H^0\bigl(\Theta, \cO_{\Theta}(\Theta)\bigr)$.

\vskip 4pt

We consider the Abel-Jacobi map
\[
\rho \colon C^{(g-1)} \lra \Theta,  \ \ \ D\mapsto \rho(D):=\eta^{\vee}(D).
\]
It is well-known \cite{Kempf73} that $\rho$ is a rational resolution of singularities. Having fixed the point $x\in C$, we introduce the following effective divisors on $C^{(g-1)}$

$$
C^{(g-2)}_x:=\Bigl\{D'+x\in C^{(g-1)}: D'\in C^{(g-2)}\Bigr\}, \\ \quad E_x := \Bigl\{D\in C^{(g-1)}: h^0\bigl(C, \omega_C(-x-D)\bigr)>0\Bigr\}.
$$
It is easily seen that for any distinct points $x, y\in C$,
$$
\rho^{*}\bigl(\Theta_{x-y}\cdot \Theta)= C^{(g-1)}_x + E_y.
$$
Indeed, if $D\in \rho^*(\Theta_{x-y}\cdot \Theta)$, then $H^0(C, D(y-x))\neq 0$, which implies that either $D-x$ is an effective divisor, that is, $D\in C^{(g-2)}_x$, or else $h^0\bigl(C, \cO(D+y)\bigr)\geq 2$, or equivalently $H^0\bigl(C, \omega_C(-x-D)\bigr)\neq 0$, in which case $D\in E_x$. Furthermore, this set-theoretic equality is in fact an equality of divisors on $C^{(g-1)}$, see also \cite[p.6]{welters86}. Furthermore, one also has the equality
$$\rho^{*}\bigl(\partial_x \Theta\bigr) = C^{(g-2)}_x + E_x,$$
which follows from previous considerations after regarding $\partial_x \Theta$ as the limit of the divisor $\Theta_{x-y}\cdot \Theta$ on $\Theta$, when the points $x$ and $y$ coalesce. We summarize this discussion as follows:

\begin{lemma}\label{lemtxty}
Suppose $x, y\in C$ are distinct points. Then we have the following equality of divisors on $\Theta$
\[
\bigl(\Theta_{x-y} + \Theta_{y-x}\bigr) \cdot \Theta = \del_x \Theta + \del_y \Theta.
\]
\end{lemma}

In what follows, we complete the proof of Theorem \ref{thm:scorza_quartic}. We keep the notation from Theorem \ref{thm:thet_quartic}.

\vskip 5pt

\noindent \emph{Proof of Theorem \ref{thm:scorza_quartic}.}
Let $\beta\colon \mathrm{Sym}^2 H^0(C,\omega_C)^{\vee}\longrightarrow \mathrm{Sym}^2 H^0(C, \omega_C)$ be as before and fix a point $(x,y)\in S(C, \eta)$. Since $h^0(C,\eta)=0$, the linear systems $|\eta(x-y)|$ and $|\eta(y-x)|$ each contain exactly one effective divisor and the sum of these two divisors is a canonical divisor $D_{x,y}=\mbox{div}\bigl(\eta(x-y)\bigr)+\mbox{div}\bigl(\eta(y-x)\bigr)$. We choose a canonical form $h_{x,y}\in H^0(C,\omega_C)$ such
that $\mbox{div}(h_{x,y})=D_{x,y}$. We shall show that up to a non-zero constant, one has
$$\beta\bigl(\phi_C(x), \phi_C(y)\bigr)=h_{x,y}^2 \in \mbox{Sym}^2 H^0(C, \omega_C).$$
This will imply that the quartic $F(C, \eta) \in \mbox{Sym}^4 H^0(C, \omega_C)$ is indeed the Scorza quartic considered in \cite{DK}.

By definition, one has that $\mu\bigl(\phi_C(x), \phi_C(y)\bigr)=\partial_x\Theta\cdot \partial_y\Theta\in H^0\bigl(\Theta, \cO_{\Theta}(2\Theta)\bigr)$. We apply Lemma \ref{lemtxty} and obtain
$$\mbox{div } \mu\bigl(\phi_C(x)\cdot \phi_C(y)\bigr)=\mbox{div}\bigl(\Theta\cdot \Theta_{x-y}+\Theta\cdot \Theta_{y-x}\bigr).$$
It follows that $$\mbox{div}\bigl(\mathrm{res}^{-1}\circ \mu\bigr)\bigl(\phi_C(x)\cdot \phi_C(y)\bigr)=\Theta_{x-y}+\Theta_{y-x}\in |2\Theta|_0.$$

\vskip 3pt
Since $(x,y) \in S( C, \eta )$, the divisors $\Theta_{x-y}$ and $\Theta_{y-x}$ both contain the origin $0\in JC$.  Furthermore, the divisors $\Theta_{x-y}$ and $\Theta_{y-x}$ have the same projectivized tangent space at $0$, namely, the span of the canonical divisor $D_{x,y}$. Hence the tangent cone at $0$ to $\Theta_{x-y} + \Theta_{y-x}$  is the quadric of rank one $q_{x,y} = 2\langle D_{x,y} \rangle$. So, up to a non-zero constant $\beta\bigl(\phi_C(x) \cdot \phi_C(y)\bigr)=h_{x,y}^2$, for all $(x, y)\in S(C ,\eta)$.
\hfill $\Box$

\subsection{The Scorza quartic hypersurfaces for vanishing theta-nulls}

We now observe that Theorem \ref{thm:thet_quartic} allows us to define the limit of the Scorza quartic $F(C,\eta)$ for a general point of the divisor $\Theta_{\mathrm{null}}$. We pick a general point $[C,\eta]\in \Theta_{\mathrm{null}}$, so that $h^0(C,\eta)=2$. In this case the theta function $\theta$ vanishes at the origin, hence $\theta_0=0$ and, from Theorem \ref{thm:thet_quartic}, we expect that $F(C,\eta)=\frac{1}{2}\theta_2^2$. In fact we can make this expectation precise and argue that the map $\beta$ in Theorem \ref{thm:thet_quartic}  can still be defined in this case.

\vskip 3pt

First observe that the map $\mu\colon \mbox{Sym}^2 H^0(C, \cO_C)\longrightarrow H^0\bigl(\Theta, \cO_{\Theta}(2\Theta)\bigr)$ defined by (\ref{eq:map_mu}) remains unchanged, and sends an element $\sum c_{ij}\del_i\del_j \in \mbox{Sym}^2 H^1 (C,\cO_C )$ to $\sum c_{ij}\del_i\theta \cdot \del_j \theta$. The map $\mbox{res}\colon H^0(JC, 2\Theta)_0\rightarrow H^0\bigl(\Theta, \cO_{\Theta}(2\Theta)\bigr)$ has a one-dimensional kernel and therefore is no longer surjective. However, $\mbox{Im}(\mbox{res})\subseteq \mbox{Im}(\mu)$,  therefore
$\sum c_{ij}\partial_i(\theta)\cdot \partial_j(\theta)$ is the restriction of a pencil of divisors in $|2\Theta|_0$ which contains the element $2\Theta$.  Since $\mbox{mult}_0(2\Theta)=4$, for each of these pencils, the map sending a divisor to its quadric tangent cone at $0$ is either undefined or constant. Since $h^0(C, \eta)= 2$, for $x,y \in C$, the quadric tangent cone at $0$ to $\del_x\Theta + \del_y\Theta$ is again $q_{x,y}$. Therefore, on a non-empty open subset of pencils, the quadric cone map is well-defined and constant and the map
\[
\beta \colon \PP\bigl(\mbox{Sym}^2 H^1(C, \cO_C)\bigr)\dashrightarrow \PP\bigl(\mbox{Sym}^2 H^0(C,\omega_C)\bigr)
\]
is well-defined as a rational map. For a general pair $(x,y)\in C$, the divisor $D_{x,y}$ is the sum of the two divisors of the pencil $|\eta |$ containing $x$ and $y$. Hence the symmetric quadrilinear form associated to the Scorza quartic (that is, to the map $\beta$) vanishes on all quadruples of the form $(x,y, z_1, z_2)$, where $z_1, z_2$ are arbitrary points of $\langle D_{x,y} \rangle$. From this, it follows:

\begin{theorem}\label{thm:limit_scorza}
The limiting Scorza quartic $F(C,\eta)\in \mathrm{Sym}^4 H^0(C,\omega_C)$ of a general point $[C, \eta]\in \Theta_{\mathrm{null}}$ is well-defined and equals twice the quadratic tangent cone to $\Theta$ at $0$.
\end{theorem}

\end{document}